\documentclass[10pt,twoside,reqno]{amsart}
\usepackage[foot]{amsaddr}
\usepackage{amssymb,amsfonts,amsmath,amsthm}
\usepackage{calrsfs}
\usepackage[mathscr]{eucal}
\usepackage{bbm}
\usepackage{stmaryrd}
\usepackage{fullpage}
\usepackage{blkarray}
\usepackage{array}
\usepackage{mathdots}

\theoremstyle{definition}
\newtheorem{thm}{Theorem}
\newtheorem*{thm*}{Theorem}
\newtheorem{prop}[thm]{Proposition} 
\newtheorem{lem}[thm]{Lemma}
\newtheorem{defi}[thm]{Definition} 
\newtheorem{cor}[thm]{Corollary}
\newtheorem{rem}[thm]{Remark}

\numberwithin{equation}{section}
\numberwithin{thm}{section}

\usepackage{enumerate}
\usepackage[shortlabels]{enumitem}

\usepackage{hyperref}
\usepackage[dvipsnames]{xcolor}
\newcommand\myshade{85}
\colorlet{mylinkcolor}{red}
\colorlet{mycitecolor}{blue}
\colorlet{myurlcolor}{Aquamarine}
\hypersetup{
	linkcolor  = mylinkcolor,
	citecolor  = mycitecolor,
	urlcolor   = myurlcolor!\myshade!black,
	colorlinks = true,
}

\usepackage{tikz}

\usetikzlibrary{arrows}

\newcommand{\ZZ}{\mathbb{Z}}

\newcommand{\CC}{\mathbb{C}}

\DeclareMathOperator{\newb}{new}

\DeclareMathOperator{\sgn}{sgn}

\DeclareMathOperator{\ct}{CT}

\newcommand{\dn}{\mathscr{D}}

\allowdisplaybreaks

\begin{document}

\date{}

\title{On certain identities involving Nahm-type sums with double poles}
\author{Shashank Kanade}
\address{Department of Mathematics, University of Denver, Denver, CO 80208}
\email{shashank.kanade@du.edu}

\author{Antun Milas}
\address{Department of Mathematics and Statistics, University at Albany - SUNY, NY 12222}
\email{amilas@albany.edu}

\author{Matthew C.\ Russell}
\address{Department of Mathematics, 
Rutgers, The State University of New Jersey,
Piscataway, NJ 08854}
\email{russell2@math.rutgers.edu}

\begin{abstract}
We prove certain Nahm-type sum representations for 
the (odd modulus) Andrews-Gordon identities, the (even modulus) Andrews-Bressoud identities, and Rogers' false theta functions. These identities are motivated on one hand by a recent work of C. Jennings-Shaffer and one of us  \cite{JenMil-double1,JenMil-double2} on double pole series, and, on the other hand, by C\'ordova, Gaiotto and Shao's work  \cite{cordova2016infrared} on defect Schur's indices. 

\end{abstract}
\maketitle

\section{Introduction and Motivation}

Nahm sums are certain $q$-hypergeometric series which have appeared in many areas including combinatorics, number theory, quantum topology, representation theory, and theoretical physics. More recently, a version of Nahm sums with ``double poles" emerged in connection to wall-crossings phenomena and 4d/2d dualities in physics. The double pole series of interest here is
 \begin{equation} \label{Nahm:d}
 \sum_{n_1,...,n_k \geq 0} \frac{q^{n_1 + \cdots + n_k+  \frac12 {\bf n} \cdot C \cdot {\bf n}^T}}{(q)^2_{n_1} \cdots (q)^2_{n_k}},
 \end{equation}
where $C$ is the incidence matrix of a graph with $k$ vertices and ${\bf n}=(n_1,...,n_k)$.  If $C$ is of ADE type, then physicists predicted that this expression is essentially Schur's index (or coefficient thereof) of a certain {\em 4d} $N=2$ Argyres-Douglas theory \cite{cordova2016schur}.  As demonstrated in the same paper, this $q$-series can be interpreted as a quantum torus-valued trace of the Kontsevich-Soibelman operator $\mathcal{O}(q)$  \cite{cordova2016schur, GaiMooNei-WCF}. 
Then powerful 4d/2d dualities allow one to obtain a new representation of (\ref{Nahm:d}) in the form of a particular vertex algebra character ({\em 2d} object), with additional Euler factors. This has led to the discovery of a new family of $q$-series identities of sum=product type. The simplest identity of this type is (here $C=A_2$)
$${(q)^2_\infty} \sum_{n_1,n_2 \geq 0} \frac{ q^{n_1+n_2+n_1 n_2}}{(q)_{n_1}^2 (q)_{n_2}^2}=\frac{1}{(q^3,q^3;q^5)_\infty},$$
where the right-hand side is the product side of the second Rogers-Ramanujan identity, which is also the vacuum character of the $(2,5)$ Virasoro minimal model.

A mathematical study of Nahm sums with higher order poles was undertaken in \cite{JenMil-double1,JenMil-double2} for specific 
 Andrews-Gordon series,  Andrews-Bressoud series, and their false theta function counterparts. In particular,  a generalization of identities discovered in \cite{cordova2016schur} was obtained.  More recently, C\'ordova, Gaiotto and Shao pushed further their method to obtain a new identity for the first Rogers-Ramanujan series (this corresponds to the non-vacuum Virasoro $(2,5)$ minimal model):
\begin{equation} \label{RR:non-vacuum}
{(q)^2_\infty} \sum_{n_1,n_2 \geq 0} \frac{(2-q^{n_1}) q^{n_1+n_2+n_1 n_2}}{(q)_{n_1}^2 (q)_{n_2}^2}=\frac{1}{(q,q^4;q^5)_\infty}.
\end{equation}
by modifying the $\mathcal{O}(q)$ operator using surface defects \cite{CorGaiSha-surfacedef}. 

The aim of this  paper is to provide a conceptual explanation of this and more general Andrews-Gordon identities, thus generalizing (\ref{RR:non-vacuum}) and \cite[Theorem 5.2]{JenMil-double1}.  We also consider other closely related series 
such as Andrews-Bressoud series and their ``false" counterparts. This way, in particular, we obtain double pole representations of characters of {\em all} $(2,2k+3)$ Virasoro minimal models 
and {\em all} $(2,4k)$ $N=1$ superconformal minimal models. 
Our main objects of study are the double-pole Nahm-type sums $\dn_{t,s}$:
	$$\sum\limits_{n_1,\dots,n_t\geq 0}\dfrac{(-w)_{n_1}q^{n_1n_2+\cdots+n_{t-1}n_t + n_1 + \cdots + n_{s-1} + a n_s+n_{s+1}+\cdots +n_k} }{(q)_{n_1}^2\cdots (q)_{n_t}^2}$$
	where we allow $ a \in \{1,2 \}$.
All identities of interest in this paper follow by easy specialization of $w$. 
Our main tool is the machinery of Bailey pairs, substantially generalizing the techniques used in \cite{JenMil-double1,JenMil-double2}. 
We also present a new approach to double pole identities based on hypergeometric summation.
	
The paper is organized as follows. 
In Section \ref{sec:prelim} we gather auxiliary results on $q$-series and basic hypergeometric summations. 
We also recall the Andrews-Gordon identities and certain identities for the unary false theta functions, 
and related $q$-difference equations. 
Results on Bailey pairs needed in the paper are presented in Section \ref{sec:Bailey}. 
In Section \ref{sec:dpns}, we first introduce the $\dn_{t,s}$ series, which are the main objects of study. 
Using  the method of quantum dilogarithm, 
we find a representation of $\dn_{t,s}(w,q)$ in the form of an ordinary Nahm sum (with single poles!) 
with some extra signs (see Propositions \ref{prop:doublepoleBailey} and \ref{prop:doublepoleBailey1}). 
In Section \ref{sec:dpnsBailey}, equipped with results from Section \ref{sec:Bailey}, we now use the Bailey pair machinery  to 
obtain a theta series type representation for $\dn_{t,s}(w,q)$ (see Proposition \ref{thm:dp-fermionic}). 
In Section \ref{sec:sum=product} we begin the analysis of various specializations in Theorem \ref{thm:dp-fermionic}. 
For an even number of summation variables specialized at $w=0$,  we get $q$-series identities for the Andrews-Gordon series (Theorem \ref{thm:doublepoleGA}), and 
Andrews-Bressoud series for $w=1$ (Theorem \ref{thm:doublepoleAB}) and $w=q^{1/2}$ (Theorem \ref{thm:doublepoleAB2}). 
For an odd number of summation variables, we prove new $q$-series representations of all unary false theta functions  
in Theorems \ref{thm:doublepoleFalseAG}, \ref{thm:doublepoleFalseAB}, \ref{thm:doublepoleFalseAB2}.  
In Section \ref{sec:alt}, we present an alternative approach to double pole identities. As an illustration of the method, we reprove several special cases from Sections \ref{sec:sum=product} and \ref{sec:ft}. We finish by outlining an agenda for future research.

{\bf Acknowledgments:} 
S.K. is currently supported by the Collaboration Grant for Mathematicians \#636937 awarded by the Simons
Foundation. A.M. was partially supported by the Collaboration Grant for Mathematicians \#709563 awarded by the Simons
Foundation.

\section{Preliminary $q$-series identities}
\label{sec:prelim}
As usual, we let $(a)_n=(a;q)_n=\prod_{i=0}^{n-1} (1-aq^{i})$ and $(a_1,...,a_k;q)_n=(a_1)_n \cdots (a_k)_n$.
Throughout, we will use the fact that if $n<0$ then 
\begin{align}
	\label{eqn:pochneg0}
	\dfrac{1}{(q)_n}=0.
\end{align}

We have the following basic relations due to Euler, see \cite[Corollary 2.2]{And-book}:
\begin{align}
	\sum_{n\geq 0}\frac{z^n}{(q)_n}=\frac{1}{(z)_{\infty}},
	\label{eqn:euler1}\\
	\sum_{n\geq 0}\frac{(-1)^nz^nq^{\binom{n}{2}}}{(q)_n}={(z)_{\infty}}
	\label{eqn:euler2}.
\end{align}
The $q$-binomial theorem \cite[Equation (II.3)]{GasRah-book} states that:
\begin{align}
	\sum_{n\geq 0}\frac{(a)_nz^n}{(q)_n}=\frac{(az)_\infty}{(z)_{\infty}}.
	\label{eqn:qbin}
\end{align}
We will also require Heine's transformations \cite[Corollary 2.3]{And-book}
\begin{align}
	\label{eqn:Heine}
	\sum_{n\geq 0} \dfrac{(a)_n(b)_n}{(q)_n(c)_n}t^n 
	=\dfrac{(b)_\infty(at)_{\infty}}{(c)_\infty(t)_\infty}
	\sum_{n\geq 0}\dfrac{(c/b)_n(t)_n}{(q)_n(at)_n}b^n
\end{align}
\begin{equation} \label{Heine}
\sum_{n \geq 0} \frac{(a)_n (b)_n z^n}{(c)_n (q)_n}=\frac{(abz/c;q)_\infty}{(z;q)_\infty} \sum_{n \geq 0} \frac{(c/a)_n (c/b)_n \left(\frac{abz}{c}\right)^n}{(c)_n (q)_n}.
\end{equation}
and Jackson's summation formula
\begin{align}
	\label{Jackson}
	\sum_{n\geq 0} \dfrac{(a)_n(b)_n}{(q)_n(c)_n}z^n 
	=\dfrac{(az;q)_\infty}{(z;q)_\infty}
	\sum_{k \geq 0}\dfrac{(a)_k (c/b)_k}{(q)_k (c)_k (az)_k}(-bz)^k q^{k(k-1)/2}.
\end{align}

We will frequently need the Jacobi triple product identity:
\begin{align}
	\sum_{n\in\ZZ}a^{\frac{n(n+1)}{2}}b^{\frac{n(n-1)}{2}}
	=(-a;ab)_\infty (-b;ab)_\infty (ab;ab)_\infty \label{eqn:jtp}.
\end{align}
Recall also Rogers' false theta function \cite{Sil-book}:
\begin{align}
	\Psi(a,b)=
	\sum_{n\in\ZZ}\sgn^*(n) a^{\frac{n(n+1)}{2}}b^{\frac{n(n-1)}{2}}
	&=\sum_{n\geq 0} a^{\frac{n(n+1)}{2}}b^{\frac{n(n-1)}{2}}(1-b^{2n+1}),
	\label{eqn:Psi}
\end{align}
where 
\begin{align*}
	\sgn^*(n)=\begin{cases}
		1 & n\geq 0\\
		-1 & n<0
	\end{cases}.	
\end{align*}

We also need the following two slight modifications of a result of Andrews, \cite[Lemma 1]{And-HKP}:
\begin{align}
	\frac{1}{(q^{\frac12}\zeta)_{\infty}(q^{\frac12}\zeta^{-1})_{\infty}}
	&=\frac{1}{(q)_{\infty}^2}
	\sum_{
		\substack{m\geq n\\n\in\ZZ}
	}(-1)^{m+n}q^{\frac{m^2+m}{2} -\frac{n^2}{2}}\zeta^n,
	\label{eqn:and1}\\
	\frac{1}{(q\zeta)_{\infty}(q\zeta^{-1})_{\infty}}
	&=\frac{1}{(q)_{\infty}^2}
	\sum_{
		\substack{m\geq n\\n\in\ZZ}
	}(-1)^{m+n}q^{\frac{m^2+m}{2} -\frac{n^2-n}{2}}\zeta^n
	(1-\zeta^{-1}).
	\label{eqn:and2}
\end{align}

We finally need the pentagon relation for the quantum dilogarithm.
If $x, y$ are non-commuting variables with $xy=qyx$ then:
\begin{align}
	(y)_{\infty}(x)_{\infty}=(x)_{\infty}(-yx)_{\infty}(y)_{\infty}.
	\label{eqn:qdilogpentagon}
\end{align}

Recall also classical Andrews-Gordon identities:
\begin{thm} \label{AG-id} For $k \geq 1$ and $0 \leq i \leq k$, we have
\begin{equation} \label{MFI:AG}
\dfrac{(q^{k-i+1}, q^{k+i+2}, q^{2k+3}\,\,;\,\,q^{2k+3})_{\infty}}{(q)_{\infty}}
=\sum_{n_1,n_2,\dotsc,n_{k} \geq 0} 
\frac{q^{N_1^2+N_2^2+\cdots+N_{k}^2+N_{k-i+1}+N_{k-i+2}+\cdots + N_{k}}}
	{ (q)_{n_1} (q)_{n_2} \cdots (q)_{n_{k-1}} (q)_{n_{k}}},
\end{equation}
where $N_t=\sum_{j \geq t} n_j$.
\end{thm}
We also have identities for false theta functions due to Bringmann and one of us \cite{BM} (essentially the same identities were discovered in the  
analysis of `tails' of colored Jones polynomials of $(2,2k)$ torus knots \cite{hajij}; see also \cite{BO,KO} for related identities). 
\begin{thm} \label{false-id} For $k \in \mathbb{N}$ and $1 \leq i \leq k$, we have 
\begin{equation} \label{MFI:FT}
\frac{1}{(q)_\infty} \sum_{n \in \mathbb{Z}} {\rm sgn}^*(n) q^{(k+1)n^2+ i n}
=\sum_{n_1,n_2,\dotsc,n_{k} \geq 0} \frac{q^{N_1^2+N_2^2+\cdots+N_{k}^2+N_{k-i+1}+N_{k-i+2}+\cdots + N_{k}}}
	{(q)_{n_{k}}^2 (q)_{n_1} (q)_{n_2} \cdots (q)_{n_{k-1}}},
\end{equation}
where $N_t=\sum_{j \geq t} n_j$ as above.
Moreover, for $i=0$, we have 
\begin{equation} \label{MFI:FT1}
\frac{1}{(q)_\infty} =\sum_{n_1,n_2,\dotsc,n_{k} \geq 0} \frac{q^{N_1^2+N_2^2+\cdots+N_{k}^2}}{(q)_{n_{k}}^2 (q)_{n_1} (q)_{n_2} \cdots (q)_{n_{k-1}}}.
\end{equation}
\end{thm}

\subsection{$q$-difference equations}
Denote for $0 \leq i \leq k$, 
$$\theta_{k,i}(x,q)=\sum_{n_1, n_2, ... , n_k \geq 0} \frac{x^{N_1+\cdots + N_k} q^{N_1^2+ \cdots N_k^2+N_{k-i+1}+ \cdots N_k }}{(q)_{n_1}(q)_{n_2} \cdots (q)_{n_k}}.$$

Then we have a well-known system of $q$-difference equations
\begin{equation} \label{AG-q-diff}
	{\boldsymbol\theta}(x)=A(x,q) {\boldsymbol\theta}(qx)
\end{equation}
where
$$
{\boldsymbol \theta}(x)=\begin{pmatrix} \theta_{k,0}(x,q) \\ \theta_{k,1} (x,q) \\ ... \\ ...  \\ \theta_{k,k}(x,q) \end{pmatrix},
\quad
A(x,q)=\begin{pmatrix} 1 & xq & ... & .... & (xq)^k \\ 1 & xq & ... & ... & ...   \\  ... & ... & ... & ... & ... \\ 1 & xq & 0 & ... & .... \\   1  & 0 &0  & ... & ...& \end{pmatrix}.$$
Proof of these recursions, more precisely their inverse relations, can be found in \cite[Theorem 7.8]{And-book},  where $\theta_{k,i}(x;q)$ are denoted by $J_{k+1,k-i+1}(0;x;q)$.

For instance, for $k=1$, with $\theta_{1,0}(x,q)=\sum_{n \geq 0} \frac{x^n q^{n^2}}{(q)_n}$ and $\theta_{1,1}(x,q)=\sum_{n \geq 0} \frac{x^n q^{n^2+n}}{(q)_n}$, we get the famous
Rogers-Ramanujan recursions:
\begin{equation} \label{RR-rec}
\begin{pmatrix} \theta_{1,0}(x,q)  \\ \theta_{1,1}(x,q) \end{pmatrix}=\begin{pmatrix} 1 & xq \\ 1 & 0 \end{pmatrix} \begin{pmatrix} \theta_{1,0}(qx,q)  \\ \theta_{1,1}(qx,q) \end{pmatrix}.
\end{equation}
We also discuss $q$-difference equations for series relevant to false theta functions. For a fixed $k$, with $0 \leq i \leq k$, we let
$$\phi_{k,i}(x,q)=\sum_{n_1,n_2,\dotsc,n_{k} \geq 0} \frac{x^{N_1+\cdots + N_k} q^{N_1^2+N_2^2+\cdots+N_{k}^2+N_{k-i+1}+ \cdots + N_{k}}}
	{(q)_{n_{k}}^2 (q)_{n_1} (q)_{n_2} \cdots (q)_{n_{k-1}}}.$$
\begin{prop} We have
\begin{equation} \label{false-q-diff}
{\boldsymbol\phi}(x)=B(x,q) {\boldsymbol\phi}(qx),
\end{equation}
where
$$
{\boldsymbol\phi}(x):=\begin{pmatrix} \phi_{k,0}(x,q) \\ \phi_{k,1} (x,q) \\ ... \\ ...  \\ \phi_{k,k}(x,q) \end{pmatrix},\quad 
B(x,q):=\begin{pmatrix} k+1 & -k(1-xq) & ... & .... & -(1-xq)(xq)^{k-1} \\ k & -(k-1)(1-xq) & ... & ... & 0  \\  ... & ... & ... & ... & ... \\ 2 & -(1-xq) & 0 & ... & 0 \\   1  & 0 &0  & ... & 0 &  \end{pmatrix}.$$
\end{prop}
\begin{proof} Our proof is only a slight modification of the inductive proof of (\ref{AG-q-diff}) given in \cite[Theorem 7.8]{And-book} so we omit most details. 
In order to prove (\ref{false-q-diff}), it suffices to check ``inverse" $q$-difference equations: 
\begin{align*}
 \phi_{k,0}(qx,q) & =\phi_{k,k}(x,q) \\ 
 (1-xq) \phi_{k,1}(qx,q) & =-\phi_{k,k-1}(x,q)+2\phi_{k,k}(x,q), \\
 (1-xq)(xq)\phi_{k,2}(xq,q) & =-\phi_{k,k-2}(x,q)+2\phi_{k,k-1}(x,q)-\phi_k(x,q) \\ 
 & \ldots  \ldots  \\
 (1-xq)(xq)^{k-1} \phi_{k,k}(xq,x) & =-\phi_{k,0}(x,q)+2\phi_{k,1}(x,q)-\phi_{k,2}(x,q) .
\end{align*}
The first equation obviously holds. It is convenient to write
$$\phi_{k,i}(x,q)=\sum_{n \geq 0} \frac{x^{k n} q^{k n^2+ i n}}{(q)_n^2} \theta_{k-1,i}(q^{2n}x,q).$$
Using this form, as in loc.cit., we check the remaining difference equations using the difference relations satisfied by $\theta_{k-1,i}$.
\end{proof}

For example, for $k=2$, we get 
\begin{equation} \label{false-rec}
\begin{pmatrix} \phi_0(x,q) \\ \phi_1(x,q) \\ \phi_2(x,q) \end{pmatrix}=\begin{pmatrix} 3 & -2(1-xq) & -xq(1-xq) \\ 2 & -(1-xq) & 0 \\ 1 & 0 & 0 \end{pmatrix} \begin{pmatrix} \phi_0(qx,q)  \\ \phi_1(qx,q) \\ \phi_2(qx,q)  \end{pmatrix}.
\end{equation}

\section{Results on Bailey pairs}
\label{sec:Bailey}

In this section, we recall various known results on Bailey pairs and 
also provide a few new ones.
All matrices we consider will be infinite matrices with row and column indices being $0,1,2,\dots$.

Keeping in mind \eqref{eqn:pochneg0}, we let $L(a)$ be the lower triangular Bailey matrix:
\begin{align*}
	[L(a)]_{r,c} = \dfrac{1}{(q)_{r-c}(aq)_{r+c}}.
\end{align*}
When we have $a=q$, we will simply write $L$ instead of $L(q)$.
This matrix is invertible \cite{AgaAndBre-lattice}:
\begin{align}\label{eqn:Linv}
	[L(a)^{-1}]_{r,c} = (-1)^{r-c}q^{\binom{r-c}{2}}\frac{(a)_{r+c}}{(q)_{r-c}}\frac{(1-aq^{2r})}{(1-a)}.
\end{align}

\begin{defi}
	We say that two sequences $\alpha_n$, $\beta_n$ $(n\in\ZZ_{\geq 0})$ form a Bailey pair with respect 
	to  $a$, if for all $n\in\ZZ_{\geq 0}$, 
	\begin{align*}
		\beta_n=\sum_{r=0}^n \dfrac{\alpha_r}{(q)_{n-r}(aq)_{n+r}}.
	\end{align*}
	In matrix notation, we have:
	\begin{align*}
		\beta = L(a)\cdot \alpha
	\end{align*}
	where $\alpha$ and $\beta$ are the infinite column vectors:
	\begin{align*}
		\alpha &= [\alpha_0,\alpha_1,\alpha_2,\dots]^T,\quad 
		\beta = [\beta_0,\beta_1,\beta_2,\dots]^T.
	\end{align*}
\end{defi}	

In order to deduce $q$-series identities, we shall employ the following well-established strategy.
We shall start with an initial (well-known) Bailey pair. This pair is then modified appropriately to arrive at the final Bailey pair.
Then, a requisite limit of the equation asserting that this final Bailey pair is indeed a Bailey pair will give us our desired 
$q$-series identities.

To achieve this, we will need to change $\beta$ vector so that:
\begin{align*}
	\beta^{\newb} = [\beta_0^{\newb},\beta_1^{\newb},\beta_2^{\newb},\dots]^T = M\cdot \beta,
\end{align*}
for suitable matrices $M$. 
We will always assume that in each row of $M$ there are only finitely many non-zero entries, so that products like $ML(a)$, $L(a)^{-1}ML(a)$, etc. make sense.
In this case, it is easy to see that we have:
\begin{align*}
	\beta^{\newb} = L(a)\cdot\alpha^{\newb},
\end{align*}
i.e., $\alpha^{\newb}$ and $\beta^{\newb}$ form a Bailey pair with respect to $a$, where:
\begin{align*}
	\alpha^{\newb} = [\alpha_0^{\newb},\alpha_1^{\newb},\alpha_2^{\newb},\dots]^T = L(a)^{-1}ML(a)\cdot\alpha.
\end{align*}

\begin{defi}
	Whenever $L(a)^{-1}ML(a)$ is well-defined, we will denote
it by $\widetilde{M}$. The choice of $a$ will be clear from context.
\end{defi}	

We have the following standard modifications, which we call ``forward moves''.
\begin{prop}
	If $\alpha_n$, $\beta_n$ $(n\in\ZZ_{\geq 0})$ form a Bailey pair with respect to $a=q$, then so do:
	\begin{align}
		\alpha_n^{\newb} &= 
		(-1)^nq^{\frac{n(n+1)}{2}}\alpha_n\tag{F $\alpha$} \label{move:Fa}\\
		\beta_n^{\newb} &= \sum_{r=0}^n  (-1)^r\dfrac{q^{\frac{r(r+1)}{2}}(q)_r}{(q)_n(q)_{n-r}}\beta_r.\tag{F $\beta$}\label{move:Fb}
	\end{align}
\end{prop}
\begin{proof}
	We let $a=q$, $\rho_1\rightarrow\infty, \rho_2\rightarrow q$ in \cite[Theorem 3.3]{And-qserbook}.
\end{proof}

\begin{lem} Let $w\in\CC$ or be a formal variable.
	If $\alpha_n$ and $\beta_n$ form a Bailey pair with respect to $a=q$, then so do:
	\begin{align}
		\alpha^{\newb}_n &= \frac{q^{\frac{n^2+n}{2}} w^n(-w^{-1}q)_n}{(-wq)_n}\alpha_n
		\tag{Fw $\alpha$}\label{move:Fwa}\\
		\beta^{\newb}_n &= \sum_{r=0}^n\frac{(-w^{-1}q)_r w^rq^{\frac{r^2+r}{2}}}{(-wq)_n(q)_{n-r}}\beta_r
		\tag{Fw $\beta$}\label{move:Fwb}.
	\end{align}
\end{lem}
\begin{proof}
	We let $a=q$, $\rho_1\rightarrow\infty, \rho_2\rightarrow -w^{-1}q$ in \cite[Theorem 3.3]{And-qserbook}.
\end{proof}

Now we shall deduce a few new transformations.

\begin{prop}\label{prop:upshift}
	Let $U$ be the up-shift matrix:
	\begin{align*}
		[U]_{r,c} =
		\begin{cases}
			1 & c=r+1, r\geq 0\\
			0 & \mathrm{otherwise}
		\end{cases}.
	\end{align*}
	We have the following formula for the entries of the matrix $\widetilde{U}$.
	For the zeroth column:
	\begin{align*}
		[\widetilde{U}]_{n,0}
		&=\begin{cases}
			\dfrac{1}{(1-q)(1-aq)} & n=0\\
			&\\
			\dfrac{-aq-q+aq^3+aq^2}{(1-q)(1-q^2)(1-aq)} & n=1\\
			&\\
			(-1)^nq^{\binom{n+1}{2}}\dfrac{(1-aq^{2n})(aq)_{n-2}}{(q)_{n+1}} & n\geq 2
		\end{cases}		
	\end{align*}
	Moreover,
	\begin{align*}
		[\widetilde{U}]_{n,n+1}
		&=\dfrac{1}{(1-aq^{2n+1})(1-aq^{2n+2})}\quad (n\geq 0)\\
		[\widetilde{U}]_{n,n}
		&=\dfrac{-(1+q)aq^{2n-1}}{(1-aq^{2n-1})(1-aq^{2n+1})}\quad (n\geq 1)\\
		[\widetilde{U}]_{n,n-1}
		&=\dfrac{a^2q^{4n-3}}{(1-aq^{2n-2})(1-aq^{2n-1})}\quad (n\geq 2).
	\end{align*}
	All other entries of $\widetilde{U}$ are zero.
\end{prop}

\begin{proof}
In what follows, we will use the following equality which is easy to establish. For $A\geq B\geq 0$,
\begin{align}
	(a)_{A-B} = (-1)^Ba^{-B}q^{-AB+\binom{B+1}{2}}\frac{(a)_A}{(a^{-1}q^{-A+1})_B}.
	\label{eqn:pochA-B}
\end{align}

Clearly, we have, using \eqref{eqn:Linv}:
\begin{align*}
	[\widetilde{U}]_{n,c} = \frac{1-aq^{2n}}{1-a}\sum_{r\geq 0}(-1)^{n-r}q^{\binom{n-r}{2}}\frac{(a)_{n+r}}{(q)_{n-r}}\frac{1}{(q)_{r+1-c}(aq)_{r+1+c}}.
\end{align*}

For the zeroth column $c=0$, we have:
\begin{align*}
	[\widetilde{U}]_{n,0} &= \frac{1-aq^{2n}}{1-a}\sum_{r\geq 0}(-1)^{n-r}q^{\binom{n-r}{2}}\frac{(a)_{n+r}}{(q)_{n-r}}\frac{1}{(q)_{r+1}(aq)_{r+1}}
	=\frac{1-aq^{2n}}{1-a}\sum_{r=0}^n(-1)^{r}\frac{(a)_{2n-r}q^{\binom{r}{2}}}{(q)_{r}(q)_{n+1-r}(aq)_{n+1-r}},
\end{align*}
where in the second equality we have made the change $r\mapsto n-r$.
Now we have, using \eqref{eqn:pochA-B}:
\begin{align*}
	(a)_{2n-r}&=(-1)^r a^{-r}q^{-2nr+\binom{r+1}{2}}\frac{(a)_{2n}}{(a^{-1}q^{-2n+1})_r}\\
	\frac{1}{(q)_{n+1-r}}&=(-1)^r q^{r}q^{(n+1)r-\binom{r+1}{2}}\frac{(q^{-1}q^{-n-1+1})_{r}}{(q)_{n+1}}\\
	\frac{1}{(aq)_{n+1-r}}&=(-1)^r a^rq^{r}q^{(n+1)r-\binom{r+1}{2}}\frac{(a^{-1}q^{-1}q^{-n-1+1})_{r}}{(aq)_{n+1}}.
\end{align*}
Combining, we get:
\begin{align*}
	[\widetilde{U}]_{n,0} 
	&=\frac{1-aq^{2n}}{1-a}\frac{(a)_{2n}}{(q)_{n+1}(aq)_{n+1}}\sum_{r=0}^n q^{3r}\frac{(q^{-n-1})_r(a^{-1}q^{-n-1})_r}{(q)_r(a^{-1}q^{-2n+1})_r}\nonumber\\
	&=\frac{(a)_{2n+1}}{(q)_{n+1}(a)_{n+2}}
	\left(-q^{3(n+1)}\frac{(q^{-n-1})_{n+1}(a^{-1}q^{-n-1})_{n+1}}{(q)_{n+1}(a^{-1}q^{-2n+1})_{n+1}}+\sum_{r=0}^{n+1} q^{3r}\frac{(q^{-n-1})_r(a^{-1}q^{-n-1})_r}{(q)_r(a^{-1}q^{-2n+1})_r} \right)\nonumber\\
	&=\frac{(a)_{2n+1}}{(q)_{n+1}(a)_{n+2}}
	\left(-q^{3(n+1)}\frac{(q^{-n-1})_{n+1}(a^{-1}q^{-n-1})_{n+1}}{(q)_{n+1}(a^{-1}q^{-2n+1})_{n+1}}+
	\frac{(a^{-1}q^{-n+2})_{\infty} (q^{-n+2})_{\infty} }{ (a^{-1}q^{-2n+1})_{\infty} (q^3)_{\infty}}
	\right)\nonumber\\
	&=\begin{cases}
		\dfrac{1}{(1-q)(1-aq)} & n=0\\
		&\\
		\dfrac{-aq-q+aq^3+aq^2}{(1-q)(1-q^2)(1-aq)} & n=1\\
		&\\
		(-1)^nq^{\binom{n+1}{2}}\dfrac{(1-aq^{2n})(aq)_{n-2}}{(q)_{n+1}} & n\geq 2
	\end{cases}		
\end{align*}
Where in the third equality, we have used \eqref{eqn:Heine} followed by \eqref{eqn:qbin}.
For columns $c\geq 1$, the procedure is similar. 
\begin{align*}
	[\widetilde{U}]_{n,c} &= \frac{1-aq^{2n}}{1-a}\sum_{r\geq 0}(-1)^{n-r}q^{\binom{n-r}{2}}\frac{(a)_{n+r}}{(q)_{n-r}}\frac{1}{(q)_{r+1-c}(aq)_{r+1+c}}\nonumber\\
	&=\frac{1-aq^{2n}}{1-a}\sum_{r=0}^{n+1-c}(-1)^{r}q^{\binom{r}{2}}\frac{(a)_{2n-r}}{(q)_{r}}\frac{1}{(q)_{n+1-c-r}(aq)_{n+1+c-r}}\nonumber\\
	&=\frac{1-aq^{2n}}{1-a}\frac{(a)_{2n}}{(q)_{n+1-c}(aq)_{n+1+c}}\sum_{r=0}^{n+1-c}q^{3r}\frac{(q^{-n+c-1})_r(a^{-1}q^{-n-c-1})_r}{(q)_r(a^{-1}q^{-2n+1})_r}\nonumber\\
	&=\frac{(a)_{2n+1}}{(q)_{n+1-c}(a)_{n+2+c}} \frac{(a^{-1}q^{-n-c+2})_{\infty}(q^{-n+c+2})_{\infty}}{(a^{-1}q^{-2n+1})_{\infty}(q^3)_{\infty}}\nonumber\\
	&=
	\begin{cases}
		\dfrac{1}{(1-aq^{2n+1})(1-aq^{2n+2})} & c=n+1\\
		&\\
		\dfrac{-(1+q)aq^{2n-1}}{(1-aq^{2n-1})(1-aq^{2n+1})} & c=n\\
		& \\
		\dfrac{a^2q^{4n-3}}{(1-aq^{2n-2})(1-aq^{2n-1})} & c=n-1\\
		& \\
		0 & \mathrm{otherwise.}
	\end{cases}		
\end{align*}
\end{proof}

The following transformation of the Bailey pairs relative to $a=q$ 
is implied by the above proposition.
\begin{cor}
	If $\alpha_n$ and $\beta_n$ ($n\in\ZZ_{\geq 0}$) form a Bailey pair relative to $a=q$, then so do:
	\begin{align}
		\tag{U $\alpha$} \label{move:Ua}
		\alpha^{\newb}_n &= f(n) \cdot \alpha_0 + 
		g(n)\cdot  \alpha_{n-1} + 
		h(n)\cdot \alpha_{n} + 
		k(n)\cdot  \alpha_{n+1},\\
		\tag{U $\beta$} \label{move:Ub}
		\beta^{\newb}_n &= \beta_{n+1},
	\end{align}
	where:
	\begin{align*}
		\alpha_{-1} =0,
	\end{align*}
	\begin{align*}
		f(n) = 
		\begin{cases}
			\dfrac{1}{(1-q)(1-q^2)} & n=0\\
			\dfrac{-q}{(1-q)^2} & n=1\\
			(-1)^n \dfrac{q^{n(n+1)/2}(1-q^{2n+1})}{(1-q)(1-q^n)(1-q^{n+1})} & n>1
		\end{cases},
		\qquad\qquad
		g(n)  = 
		\begin{cases}
			0 & n=0,1\\
			\dfrac{q^{4n-1}}{(1-q^{2n})(1-q^{2n-1})} & n>1
		\end{cases},
	\end{align*}
	\begin{align*}
		h(n)  = 
		\begin{cases}
			0 & n=0\\
		\dfrac{-q^{2n}(1+q)}{(1-q^{2n})(1-q^{2n+2})} & n>0
		\end{cases},
		\qquad\qquad
		k(n)  = \dfrac{1}{(1-q^{2n+2})(1-q^{2n+3})}.
	\end{align*}
\end{cor}

\begin{rem}\label{rem:LovjoyOsburn-upshift}
	A transformation very similar to $(\alpha_n^{\newb},\beta_n^{\newb})$ was found in \cite[Theorem 1.2]{LovOsb-mockdouble}.
	This transformation says that if $\alpha,\beta$ form a Bailey pair relative to $a=1$ such that $\alpha_0=\beta_0=0$,
	then, the following form a
	Bailey pair relative to $a=q$:
	\begin{align*}
	 \alpha^{\newb}_n = \dfrac{1}{1-q}	\left( \dfrac{\alpha_{n+1}}{1-q^{2n+2}} - \dfrac{q^{2n}\alpha_n}{1-q^{2n}} \right),\quad \beta^{\newb}_n &= \beta_{n+1}.
	\end{align*}
\end{rem}

\begin{prop}
	Let $D$ be the infinite diagonal matrix:
	\begin{align*}
		D=\mathrm{Diag}\{q^n\,\vert\, n\geq 0\}.
	\end{align*}		
	Then, we have that:
	\begin{align}
		[\widetilde{D}]_{r,c}
		=
		\begin{cases}
			q^r & r=c\\
			a^{r-c-1}q^{r^2-r-c^2}(aq^{2r}-1) & r>c \\
			0 & \mathrm{otherwise}.
		\end{cases}\label{eqn:Dtilde}
	\end{align}
\end{prop}
\begin{proof}
	Letting $b \to \infty$ in \cite[Equations (2.4)-(2.5)]{Lov-lattice}
	we see that if $\alpha_n$, $\beta_n$ $(n\in\ZZ_{\geq 0})$
	form a Bailey pair with respect to base $a$ then the following form a Bailey pair with respect to base $aq$:
	\begin{align*}
		\alpha^\ast_n =\frac{1-aq^{2n+1}}{1-aq}q^{-n}\sum_{r=0}^n\alpha_r, \quad \beta^\ast_n = q^{-n}\beta_n.
	\end{align*}
	In matrix notation, this means that for all vectors $\alpha$,
	\begin{align*}
		L(aq)^{-1}D^{-1}L(a)\cdot \alpha = M\cdot \alpha
	\end{align*}
	where 
	\begin{align*}
		[M]_{r,c}=
		\begin{cases}
			\dfrac{1-aq^{2r+1}}{1-aq}q^{-r} & 0\leq r\leq c,\\
			&\\
			0 & \mathrm{otherwise}.
		\end{cases}			
	\end{align*}
	However, since $L(aq)^{-1}D^{-1}L(a)$ is an invertible matrix (a product of three lower triangular matrices each having non-zero diagonal entries),
	we may in fact conclude that $L(aq)^{-1}D^{-1}L(a)=M$.

	Further, letting $k=1$ and $d_1\rightarrow 0$ in \cite[Theorem 2.3]{Lov-lattice} we see that if $\alpha_n$, $\beta_n$ $(n\in\ZZ_{\geq 0})$
	form a Bailey pair with respect to base $a$ then the following form a Bailey pair with respect to base $aq$:
	\begin{align*}
		\alpha^\ast_n &=\sum_{r=0}^na^{n-r}q^{n^2-r^2}\alpha_r,
		\quad 
		\beta^\ast_n = \beta_n.
	\end{align*}
	By the same logic as above, this implies that:
	\begin{align*}
		[L(aq)^{-1}L(a)]_{r,c} = 
		\begin{cases}
			\dfrac{1-aq^{2r+1}}{1-a}a^{r-c}q^{r^2-c^2} & 0\leq r\leq c \\
			&\\
			0 & \mathrm{otherwise}.
		\end{cases}
	\end{align*}
	It can now be checked by direct multiplication that
	we have the following inverse of $L(aq)^{-1}L(a)$:
	\begin{align*}
		[L(a)^{-1}L(aq)]_{r,c}=
		\begin{cases}
			\dfrac{1-aq}{1-aq^{2r+1}} &  r=c\,\, (c\geq 0)\\
			&\\
			-\dfrac{(1-aq)aq^{2r-1}}{1-aq^{2r-1}} &  r=c+1\,\, (c\geq 0)\\
			&\\
			0 &\mathrm{otherwise}
		\end{cases}			
	\end{align*}
	We may now find by direct multiplication that:
	\begin{align*}
		[L(a)^{-1}D^{-1}L(a)]_{r,c}
		&=[L(a)^{-1}L(aq)\cdot L(aq)^{-1}D^{-1}L(a)]_{r,c}\nonumber\\
		&=\sum_{j=c}^r[L(a)^{-1}L(aq)]_{r,j}\cdot [L(aq)^{-1}D^{-1}L(a)]_{j,c}\nonumber\\
		&=\begin{cases}
			q^{-r} & 0\leq r=c\\
			q^{-r}(1-aq^{2r}) & r>c\\
			0 &\mathrm{otherwise}
		\end{cases}			
	\end{align*}
	One may again verify by direct calculation that this is inverse to the matrix given in 
	\eqref{eqn:Dtilde}.
\end{proof}

\begin{rem}
	Clearly, the result above implies a certain transformation of Bailey pairs relative to $a$
	where $\beta$ and $\beta^{\newb}$ are related by the diagonal matrix $D$.
	A transformation of a very similar flavour, where $\beta^{\newb}$ and $\beta$ are again related by a diagonal matrix, namely,
	$\beta^{\newb}_n = \beta_n/(1-q^{2n+1})$, was discovered in \cite[Theorem 1.3]{LovOsb-mockdouble}.
\end{rem}	

\begin{prop}
	Let $I$ be the infinite Identity matrix. Define:
	\begin{align*}
		S = U\cdot(I-D)^2=(I-qD)^2\cdot U=(I-qD)\cdot U\cdot (I-D).
	\end{align*}
	Equivalently,
	\begin{align*}
		[S]_{r,c} =
		\begin{cases}
			(1-q^{r+1})^2 & c=r+1, r\geq 0\\
			0 & \mathrm{otherwise}
		\end{cases}.
	\end{align*}
	Then, $\widetilde{S}$ is the tri-diagonal matrix given by:
	\begin{align*}
		[\widetilde{S}]_{r,c}=
		\begin{cases}
			\dfrac{q^{2r-1}(1-aq^{r-1})^2}{(1-aq^{2r-2})(1-aq^{2r-1})} & c=r-1, r\geq 1  \\
			& \\
			\dfrac{q^r(2aq^{2r}-aq^r-q^{r+1}-aq^{r-1}-q^{r}+2)}{(1-aq^{2r-1})(1-aq^{2r+1})} & c=r, r\geq 0 \\
			&\\
			\dfrac{(1-q^{r+1})^2}{(1-aq^{2r+1})(1-aq^{2r+2})} & c=r+1, r\geq 0\\
			&\\
			0 & \mathrm{otherwise}.
		\end{cases}			
	\end{align*}
\end{prop}
\begin{proof}
	For convenience, let us denote 
	\begin{align*}
		U_1 = U\cdot(I-D).
	\end{align*}
	We first calculate $\widetilde{U_1}=
	\widetilde{U}\cdot (I-\widetilde{D})$.
	Now, note that $\widetilde{U}$ has non-zero entries in the first column, but it is otherwise a tri-diagonal matrix.
	However, note that the first row of $I-\widetilde{D}$ is entirely $0$. 
	Thus, we may completely ignore the non-zero entries in the first column of $\widetilde{U}$.
	
	Let $r\geq 2$.
	Then we have:
	\begin{align*}
		&[\widetilde{U_1}]_{r,c} 
		= [\widetilde{U}]_{r,r-1}\cdot  [I-\widetilde{D}]_{r-1,c}
		+[\widetilde{U}]_{r,r}\cdot  [I-\widetilde{D}]_{r,c} 
		+[\widetilde{U}]_{r,r+1}\cdot  [I-\widetilde{D}]_{r+1,c}.
	\end{align*}
	Clearly, this expression is $0$ if $c>r+1$.
	If $c=r+1$, we may directly substitute various formulas deduced above to get:
	\begin{align*}
		[\widetilde{U_1}]_{r,r+1} &= [\widetilde{U}]_{r,r+1}\cdot  [I-\widetilde{D}]_{r+1,r+1}
		=\dfrac{1-q^{r+1}}{(1-aq^{2r+1})(1-aq^{2r+2})}.
	\end{align*}		
	Similarly, for $c=r$ we get:
	\begin{align*}
		[\widetilde{U_1}]_{r,r} 
		&= [\widetilde{U}]_{r,r}\cdot  [I-\widetilde{D}]_{r,r}
		+[\widetilde{U}]_{r,r+1}\cdot  [I-\widetilde{D}]_{r+1,r}\nonumber\\
		&=\dfrac{-(1+q)aq^{2r-1}}{(1-aq^{2r-1})(1-aq^{2r+1})} (1-q^r) + 
		\dfrac{q^{r}(1-aq^{2r+2})}{(1-aq^{2r+1})(1-aq^{2r+2})}\nonumber\\
		&=\dfrac{aq^{3r}-aq^{2r}-aq^{2r-1}+q^r}{(1-aq^{2r-1})(1-aq^{2r+1})}.
	\end{align*}		
	For $c=r-1$ we get:
	\begin{align*}
		[\widetilde{U_1}]_{r,r-1} &= [\widetilde{U}]_{r,r-1}\cdot  [I-\widetilde{D}]_{r-1,r-1}
		+[\widetilde{U}]_{r,r}\cdot  [I-\widetilde{D}]_{r,r-1}
		+[\widetilde{U}]_{r,r+1}\cdot  [I-\widetilde{D}]_{r+1,r-1}\nonumber\\
		&=\frac{a^2q^{4r-3}(1-q^{r-1})}{(1-aq^{2r-2})(1-aq^{2r-1})}
		-\frac{(1+q)aq^{2r-1}(1-aq^{2r})q^{r-1}}{(1-aq^{2r-1})(1-aq^{2r+1})}
		+\frac{aq^{3r-1}(1-aq^{2r+2})}{(1-aq^{2r+1})(1-aq^{2r+2})}\nonumber\\
		&=\frac{aq^{3r-2}(aq^{r-1}-1)}{(1-aq^{2r-2})(1-aq^{2r-1})}.
	\end{align*}	
	If $c<r-1$ we have:
	\begin{align*}
		[\widetilde{U_1}]_{r,c} 
		&= [\widetilde{U}]_{r,r-1}\cdot  [I-\widetilde{D}]_{r-1,c}
		+[\widetilde{U}]_{r,r}\cdot  [I-\widetilde{D}]_{r,c}
		+[\widetilde{U}]_{r,r+1}\cdot  [I-\widetilde{D}]_{r+1,c}\nonumber\\
		&= \frac{a^2q^{4r-3}a^{r-c-2}q^{r^2-3r+2-c^2}(1-aq^{2r-2})}{(1-aq^{2r-2})(1-aq^{2r-1})}
		- \frac{(1+q)aq^{2r-1}a^{r-c-1}q^{r^2-r-c^2}(1-aq^{2r})}{(1-aq^{2r-1})(1-aq^{2r+1})}\nonumber\\
		&\quad+ \frac{a^{r-c}q^{r^2+r-c^2}(1-aq^{2r+2})}{(1-aq^{2r+1})(1-aq^{2r+2})}\nonumber\\
		&=0.
	\end{align*}	

	One may now explicitly calculate the entries for rows $r=0, 1$ and see that they follow the same pattern.
	
	Concluding, we have that the matrix $\widetilde{U_1}$ is a tri-diagonal matrix with:
	\begin{align*}
		[\widetilde{U_1}]_{r,c}
		=\begin{cases}
			\dfrac{aq^{3r-2}(aq^{r-1}-1)}{(1-aq^{2r-2})(1-aq^{2r-1})} & c=r-1, r\geq 1,\\
			& \\
			\dfrac{aq^{3r}-aq^{2r}-aq^{2r-1}+q^r}{(1-aq^{2r-1})(1-aq^{2r+1})} & c=r, r\geq 0,\\		
			& \\
			\dfrac{1-q^{r+1}}{(1-aq^{2r+1})(1-aq^{2r+2})} & c=r+1, r\geq 0,\\
			& \\
			0 & \mathrm{otherwise}.
		\end{cases}			
	\end{align*}

	Now, $\widetilde{S}=(I-q\widetilde{D})\cdot \widetilde{U_1}$.
	
	It is easy to figure out the zeroth column of $\widetilde{S}$ by a direct calculation.
	So, let $c\geq 1$.
	We have:
	\begin{align*}
		[\widetilde{S}]_{r,c}=
		[I-q\widetilde{D}]_{r,c+1}\cdot [\widetilde{U_1}]_{c+1,c}
		+[I-q\widetilde{D}]_{r,c}\cdot [\widetilde{U_1}]_{c,c}
		+[I-q\widetilde{D}]_{r,c-1}\cdot [\widetilde{U_1}]_{c-1,c}.
	\end{align*}
	This expression is clearly $0$ if $r<c-1$.
	If $r=c-1$, we get:
	\begin{align*}
		[\widetilde{S}]_{c-1,c}
		&=[I-q\widetilde{D}]_{c-1,c-1}\cdot [\widetilde{U_1}]_{c-1,c}
		= \frac{(1-q^c)^2}{(1-q^{2c-1})(1-aq^{2c})}.
	\end{align*}
	If $r=c$, we get:
	\begin{align*}
	[\widetilde{S}]_{c,c}
	&=[I-q\widetilde{D}]_{c,c}\cdot [\widetilde{U_1}]_{c,c}
	+[I-q\widetilde{D}]_{c,c-1}\cdot [\widetilde{U_1}]_{c-1,c}\nonumber\\
	&={\frac { \left( 1-{q}^{c+1} \right)  \left( a{q}^{3c}-a{q}^{2c}-a{
		q}^{2c-1}+{q}^{c} \right) }{ \left( 1-a{q}^{2c-1} \right)  \left( 
		1-a{q}^{2c+1} \right) }}+{\frac {{q}^{{c}^{2}-c- \left( c-1 \right) 
		^{2}+1} \left( 1-{q}^{c} \right) }{1-a{q}^{2c-1}}}\nonumber\\
	&={\frac {{q}^{c}(2a{q}^{2c}-a{q}^{c}-{q}^{c+1}-a{q}^{c-1}-{q}^{
		c}+2)}{ 
		\left( 1-a{q}^{2c-1} \right)  
		\left( 1-a{q}^{2c+1} \right)}
	}.
	\end{align*}
	If $r=c+1$, we get:
	\begin{align*}
	[\widetilde{S}]_{c+1,c}
	&=[I-q\widetilde{D}]_{c+1,c+1}\cdot [\widetilde{U_1}]_{c+1,c}
	+[I-q\widetilde{D}]_{c+1,c}\cdot [\widetilde{U_1}]_{c,c}
	+[I-q\widetilde{D}]_{c+1,c-1}\cdot [\widetilde{U_1}]_{c-1,c}\nonumber\\
	&={\frac { \left( 1-{q}^{c+2} \right) a{q}^{3c+1} \left( a{q}^{c}-1
	\right) }{ \left( 1-a{q}^{2c} \right)  \left( 1-a{q}^{2c+1}
	\right) }}+
	{\frac {{q}^{ \left( c+1 \right) ^{2}-c-{c}^{2}} \left( 1-
   a{q}^{2c+2} \right)  \left( a{q}^{3c}-a{q}^{2c}-a{q}^{2c-1}+{q
   }^{c} \right) }{ \left( 1-a{q}^{2c-1} \right)  \left( 1-a{q}^{2c+1
   } \right) }}\nonumber\\
   &\qquad +{\frac {a{q}^{ \left( c+1 \right) ^{2}-c- \left( c-1
	\right) ^{2}} \left( 1-a{q}^{2c+2} \right)  \left( 1-{q}^{c}
	\right) }{ \left( 1-a{q}^{2c-1} \right)  \left( 1-a{q}^{2c}
	\right) }}\nonumber\\
   &=\dfrac{q^{2c+1}(1-aq^{c})^2}{(1-aq^{2c})(1-aq^{2c+1})}.
	\end{align*}
	If $r>c+1$, we get:
	\begin{align*}
		[\widetilde{S}]_{r,c}
		&=[I-q\widetilde{D}]_{r,c+1}\cdot [\widetilde{U_1}]_{c+1,c}
		+[I-q\widetilde{D}]_{r,c}\cdot [\widetilde{U_1}]_{c,c}
		+[I-q\widetilde{D}]_{r,c-1}\cdot [\widetilde{U_1}]_{c-1,c}\nonumber\\
		&={\frac {{a}^{r-c-2}{q}^{{r}^{2}-r- \left( c+1 \right) ^{2}+1} \left( 1
		-a{q}^{2r} \right) a{q}^{3c+1} \left( a{q}^{c}-1 \right) }{
		 \left( 1-a{q}^{2c} \right)  \left( 1-a{q}^{2c+1} \right) }}\nonumber\\
		 &+{\frac {{a}^{r-c-1}{q}^{-{c}^{2}+{r}^{2}-r+1} \left( 1-a{q}^{2r}
		 \right)  \left( a{q}^{3c}-a{q}^{2c}-a{q}^{2c-1}+{q}^{c}
		 \right) }{ \left( 1-a{q}^{2c-1} \right)  \left( 1-a{q}^{2c+1}
		 \right) }}+{\frac {{a}^{r-c}{q}^{{r}^{2}-r- \left( c-1 \right) ^{2}+1
		} \left( 1-a{q}^{2r} \right)  \left( 1-{q}^{c} \right) }{ \left( 1-a
		{q}^{2c-1} \right)  \left( 1-a{q}^{2c} \right) }}\nonumber \\
		&=0.
	\end{align*}	
\end{proof}	

We now record the transformation of Bailey pairs relative to $a=q$ implied by the above.
\begin{cor}
	If $\alpha_n$ and $\beta_n$ ($n\in\ZZ_{\geq 0}$) form a Bailey pair w.r.t. $a=q$, then so do:
	\begin{align}
		\alpha^{\newb}_0&=
		\frac{\alpha_0}{1+q} + \frac{1-q}{(1+q)(1-q^3)}\alpha_1 \nonumber\\
	\alpha^{\newb}_n&=			
		\dfrac{q^{2n-1}(1-q^n)\cdot \alpha_{n-1}}{(1-q^{2n-1})(1+q^{n})}  +
		\dfrac{2q^n\cdot  \alpha_n}{(1+q^{n})(1+q^{n+1})} +
		\dfrac{(1-q^{n+1})\cdot \alpha_{n+1}}{(1+q^{n+1})(1-q^{2n+3})},
		\qquad (n>0) \tag{S $\alpha$}\label{move:Sa}\\
		\tag{S $\beta$} \label{move:Sb}
	 	\beta^{\newb}_n &= (1-q^{n+1})^2\beta_{n+1}.
	\end{align}
\end{cor}

\section{Nahm-type sums with double poles}
\label{sec:dpns}
We shall work with the following Nahm-type sums involving squares of the Pochhammer symbols that appear in the 
denominators of their summands.

\begin{defi}
For $t\geq 1, 1\leq s\leq t+1$ we define
\begin{align}
\dn_{t,s}(w,q)=	
	\begin{cases}
	\sum\limits_{n_1,\dots,n_t\geq 0}\dfrac{(-w)_{n_1}q^{n_1n_2+\cdots+n_{t-1}n_t + n_1 + \cdots + n_{s-1} + 2n_s+n_{s+1}+\cdots +n_t} }{(q)_{n_1}^2\cdots (q)_{n_t}^2}
	& 1\leq s\leq t \\
	& \\
	\sum\limits_{n_1,\dots,n_t\geq 0}\dfrac{(-w)_{n_1}q^{n_1n_2+\cdots+n_{t-1}n_t + n_1 + \cdots + n_t} }{(q)_{n_1}^2\cdots (q)_{n_t}^2}
	& s=t+1.
	\end{cases}
\end{align}
\end{defi}

Our first aim is to now rewrite these expressions so that they become amenable to the Bailey machinery.
Note that the case $\dn_{t,t+1}(w,q)$ ($t\geq 2$) was already handled in \cite{JenMil-double2}.  Observe also for $1 \leq s \leq t$, $ \dn_{t,s}(0,q)=\dn_{t,t+1-s}(0,q)$.

\begin{prop}
	\label{prop:doublepoleBailey}
	We have the following equality for $t\geq 2, 2\leq s\leq t$:
	\begin{align}
		&\dn_{t,s}(w,q)	\nonumber\\
		&=\sum_{
			\substack{
				m_1,\cdots,m_{t-1}\geq 0
			}
			}
		\dfrac{(-1)^{\sum_{j=2}^{t-1} m_j}
			\cdot (1-q^{m_{s-1}+1})\cdot 
			q^{\sum_{j=1}^{t-1} \frac{m_j^2+m_j}{2}}
			\cdot w^{m_1}(-w^{-1}q)_{m_1}}
			{(q)_{\infty}^{t}\cdot \,\,(q)_{m_1}\cdot 
			(q)_{m_1-m_2}\cdots (q)_{m_{s-2}-m_{s-1}}(q)_{m_{s-1}-m_{s}+1}(q)_{m_{s}-m_{s+1}}\cdots (q)_{m_{t-2}-m_{t-1}} }.
		\label{eqn:DBailey}
	\end{align}	
	Here, if $s=t-1$, the final Pochhammer in the denominator of the right-hand side is $(q)_{m_{t-2}-m_{t-1}+1}$.
	If $s=t$, the denominator on the right-hand side is simply $(q)^t_\infty(q)_{m_1}(q)_{m_1-m_2}\cdots(q)_{m_{t-2}-m_{t-1}}$.

	For $s=t+1$, we have 
	\begin{align}
		&\dn_{t,t+1}(w,q)	
		=\sum_{
			\substack{
				m_1,\cdots,m_{t-1}\geq 0
			}
			}
		\dfrac{(-1)^{\sum_{j=2}^{t-1} m_j}
			\cdot 
			q^{\sum_{j=1}^{t-1} \frac{m_j^2+m_j}{2}}
			\cdot w^{m_1}(-w^{-1}q)_{m_1}}
			{(q)_{\infty}^{t}\cdot \,\,(q)_{m_1}\cdot 
			(q)_{m_1-m_2}\cdots (q)_{m_{t-2}-m_{t-1}} }.
		\label{eqn:DBaileylast}
	\end{align}	
\end{prop}	
\begin{proof}
The $s=t+1$ case is already handled in \cite[Lemma 4.2]{JenMil-double2}. 
The proof for $2\leq s\leq t$ proceeds similarly, with appropriate adjustments.

Let $\zeta_j$ be non-commuting variables such that $\zeta_j\zeta_{j+1}=q\zeta_{j+1}\zeta_j$ 
and $\zeta_i\zeta_j=\zeta_j\zeta_i$ whenever $|i-j|>1$.
We have, analogously to \cite[Proposition 4.1]{JenMil-double2}, 
\begin{align*}
	&\sum_{n_1,\dots,n_t\geq 0}\dfrac{(-w)_{n_1}q^{n_1n_2+\cdots+n_{t-1}n_t + n_1 + \cdots + n_{s-1} + 2n_s+n_{s+1}+\cdots +n_t} }{(q)_{n_1}^2\cdots (q)_{n_t}^2}
	\nonumber\\
	&=\ct_{\zeta_1,\cdots,\zeta_t} \dfrac{(-wq^{h_1}\zeta_1)_{\infty}}{(q^{h_1}\zeta_1)_{\infty}}
	\left(\prod_{j=2}^t{(q^{h_j}\zeta_j)_\infty^{-1}}{(q^{h_{j-1}}\zeta_{j-1}^{-1})^{-1}_\infty}\right)
	\frac{1}{(q^{h_{t}}\zeta_{t}^{-1})_\infty}
\end{align*}
where $h_j=\frac{1}{2}$ for $j\neq s$ and $h_s=1$. 
Further, $\ct_{\zeta_1,\cdots,\zeta_t}$ denotes the constant term with respect to $\zeta_1,\dots,\zeta_t$.
Now we manipulate the right-hand side expression (without the constant term operation).

For $j=2,\cdots,k$, we change:
\begin{align*}
	{(q^{h_j}\zeta_j)^{-1}_\infty}{(q^{h_{j-1}}\zeta_{j-1}^{-1})^{-1}_\infty}	
	&=(q^{h_{j-1}}\zeta_{j-1}^{-1})_\infty^{-1}
	(-q^{h_{j-1}+h_j}\zeta_{j-1}^{-1}\zeta_j)_\infty^{-1}
	(q^{h_{j}}\zeta_{j})_\infty^{-1}\\
	&=(q^{h_{j-1}}\zeta_{j-1}^{-1})_\infty^{-1}
	(-q^{h_{j-1}+h_j-1}\zeta_j\zeta_{j-1}^{-1})_\infty^{-1}
	(q^{h_{j}}\zeta_{j})_\infty^{-1}.
\end{align*}

We see:
\begin{align*}
	&\dfrac{(-wq^{h_1}\zeta_1)_{\infty}}{(q^{h_1}\zeta_1)_{\infty}}
	\left(\prod_{j=2}^t{(q^{h_j}\zeta_j)_\infty^{-1}}{(q^{h_{j-1}}\zeta_{j-1}^{-1})^{-1}_\infty}\right)
	\frac{1}{(q^{h_{t}}\zeta_{t}^{-1})_\infty}\\
	&=\dfrac{(-wq^{h_1}\zeta_1)_{\infty}}{(q^{h_1}\zeta_1)_{\infty}}
	\frac{1}{(q^{h_1}\zeta_1^{-1})_{\infty}}
	\left(\prod_{j=2}^{t}
	(-q^{h_{j-1}+h_j-1}\zeta_j\zeta_{j-1}^{-1})_\infty^{-1}
	(q^{h_{j}}\zeta_{j})_\infty^{-1}
	(q^{h_{j}}\zeta_{j}^{-1})_\infty^{-1}
	\right).
\end{align*}

Now we use the following expansions.
We expand $\frac{(-wq^{h_1}\zeta_1)_{\infty}}{(q^{h_1}\zeta_1)_{\infty}}$
factor with the summation variable $r_1\geq 0$ using \eqref{eqn:qbin}.
We expand $\frac{1}{(q^{h_1}\zeta_1)_{\infty}}$ using variable $r_2\geq 0$ using \eqref{eqn:euler1}.
Each $(-q^{h_{j-1}+h_j-1}\zeta_j\zeta_{j-1}^{-1})_\infty^{-1}$ is also expanded using \eqref{eqn:euler1}, and 
we use summation variable $\ell_j\geq 0$.
Factors $(q^{h_{j}}\zeta_{j})_\infty^{-1}
(q^{h_{j}}\zeta_{j}^{-1})_\infty^{-1}$ are expanded using \eqref{eqn:and1} and \eqref{eqn:and2} as appropriate
and using expansion variables $m_j\geq n_j \in \ZZ$.
We get that the previous expression equals: \begin{align*}
\sum_{
	\substack{
		r\geq 0\\
		\forall j,\,\, \ell_j\geq 0,\\
		\forall j,\,\, m_j\geq n_j\in\ZZ
	}
	}
	&
	\frac{1}{(q)^{2k-2}_{\infty}}
	\dfrac{(-1)^{\sum_{j=2}^t(\ell_j+m_j+n_j)}
	\cdot 
	q^{\frac{r_1+r_2}{2}+\frac{n_s}{2}+\sum_{j=2}^t\left( \frac{m_j^2+m_j}{2} - \frac{n_j^2}{2}
	 + \ell_j(h_{j-1}+h_j-1)\right) }\cdot (-w)_{r_1}}
	{(q)_{r_1}(q)_{r_2}\,\,(q)_{\ell_2}\cdots(q)_{\ell_t}}
	\nonumber\\
	&\times \zeta_1^{r_1-r_2}
	\left(\prod_{j=2}^t (\zeta_j\zeta_{j-1}^{-1})^{\ell_j}\zeta_j^{n_j}(1-\zeta_s^{-1})^{\delta_{j=s}}\right),
\end{align*}	
where $\delta_{j=s}$ is $1$ if $j=s$ and $0$ otherwise.
Let us denote $\ell_{t+1}=0$.

Note that the term $h_{j-1}+h_j-1$ is $0$ for $j\neq s, s+1$ ($2\leq j\leq t$),
is $\frac{1}{2}$ for $j=s$ and also for $j=s+1$ whenever $s+1\leq k$.
Thus we see that 
\begin{align*}
	\sum_{j=2}^t \ell_j(h_{j-1}+h_j-1) = \dfrac{\ell_s+\ell_{s+1}}{2}
\end{align*}
with the convention that $\ell_{t+1}=0$, which comes into play when $s=k$.

Additionally, 
\begin{align*}
	\zeta_1^{r_1-r_2}
	\left(\prod_{j=2}^t (\zeta_j\zeta_{j-1}^{-1})^{\ell_j}\zeta_j^{n_j}(1-\zeta_s^{-1})^{\delta_{j=s}}\right)
	&=q^{\sum_{j=2}^t \frac{\ell_j(\ell_j+1)}{2}}\zeta_1^{r_1-r_2}
	\left(\prod_{j=2}^t (\zeta_{j-1}^{-1}\zeta_j)^{\ell_j}\zeta_j^{n_j}(1-\zeta_s^{-1})^{\delta_{j=s}}\right)\\
	&=q^{\sum_{j=2}^t \frac{\ell_j(\ell_j+1)}{2}}
	\zeta_1^{r_1-r_2-l_1}
	\left(\prod_{j=2}^t \zeta_j^{n_j+\ell_j-\ell_{j+1}}(1-\zeta_s^{-1})^{\delta_{j=s}}\right)
\end{align*}
where we again follow the convention that $\ell_{t+1}=0$.
We thus get:
\begin{align*}
	&\frac{1}{(q)^{2k-2}_{\infty}}
	\sum_{
		\substack{
			r_1, r_2\geq 0\\
			\forall j,\,\, \ell_j\geq 0,\\
			\forall j,\,\, m_j\geq n_j\in\ZZ
		}
		}
		\dfrac{(-1)^{\sum_{j=2}^t(\ell_j+m_j+n_j)}
		\cdot 
		q^{\frac{r_1+r_2}{2}+\frac{n_s}{2}+\frac{\ell_s+\ell_{s+1}}{2}+\sum_{j=2}^t\left( \frac{m_j^2+m_j}{2} - \frac{n_j^2}{2}
		 + \frac{\ell_j(\ell_j+1)}{2}\right) }\cdot (-w)_{r_1}}
		{(q)_{r_1}(q)_{r_2}\,\,(q)_{\ell_2}\cdots(q)_{\ell_t}}
		\nonumber\\
		&\qquad\qquad\qquad\times \zeta_1^{r_1-r_2-\ell_1}
		\left(\prod_{j=2}^t \zeta_j^{n_j+\ell_j-\ell_{j+1}}(1-\zeta_s^{-1})^{\delta_{j=s}}\right)
		\\
	&=\frac{1}{(q)^{2k-2}_{\infty}}
	\sum_{
		\substack{
			r_1, r_2\geq 0\\
			\forall j,\,\, \ell_j, m_j\geq 0,\\
			\forall j,\,\, n_j\in\ZZ
		}
		}
		\dfrac{(-1)^{\sum_{j=2}^t(\ell_j+m_j)}
		\cdot 
		q^{\frac{r_1+r_2}{2}+\frac{n_s}{2}+\frac{\ell_s+\ell_{s+1}}{2}+\sum_{j=2}^t\left( \frac{m_j^2+m_j}{2} + n_jm_j + \frac{n_j}{2}
		 + \frac{\ell_j(\ell_j+1)}{2}\right) }\cdot (-w)_{r_1}}
		{(q)_{r_1}(q)_{r_2}\,\,(q)_{\ell_2}\cdots(q)_{\ell_t}}
		\nonumber\\
		&\qquad\qquad\qquad\times \zeta_1^{r_1-r_2-\ell_1}
		\left(\prod_{j=2}^t \zeta_j^{n_j+\ell_j-\ell_{j+1}}(1-\zeta_s^{-1})^{\delta_{j=s}}\right)
\end{align*}	

Now we distribute the term $(1-\zeta_s^{-1})$ and get two summations.
For the first summation, we get the constant term by setting
$2\leq j\leq k$:
\begin{align*}
	r_1&=r_2+\ell_2\\
	n_j&=\ell_{j+1}-\ell_j
\end{align*}
For the second term, the only change is:
\begin{align*}
	n_s&=\ell_{s+1}-\ell_s+1
\end{align*}
For convenience, in both cases, we also replace $r_2$ by $r$.

The first term arising from the $1$ in $(1-\zeta_s^{-1})$ becomes:
\begin{align*}
		&\sum_{
		\substack{
			r\geq 0\\
			\forall j,\,\, \ell_j,m_j\geq 0
		}
		}
		\dfrac{(-1)^{\sum_{j=2}^t(\ell_{j}+m_j)}
		\cdot 
		q^{r+\frac{\ell_2}{2}+\ell_{s+1}+\sum_{j=2}^t\left( (\ell_{j+1}-\ell_j)m_j+\frac{m_j^2+m_j}{2}+\frac{\ell_{j+1}-\ell_j}{2}
		 + \frac{\ell_j(\ell_j+1)}{2}\right) }\cdot (-w)_{r+\ell_2}}
		{(q)^{2k-2}_{\infty}\cdot (q)_{r}(q)_{r+\ell_2}\cdot (q)_{\ell_2}\cdots(q)_{\ell_t}}\\
		&=
		\sum_{
		\substack{
			r\geq 0\\
			\forall j,\,\, \ell_j,m_j\geq 0
		}
		}
		\dfrac{(-1)^{\sum_{j=2}^t(\ell_{j}+m_j)}
		\cdot 
		q^{r+\ell_{s+1}+\sum_{j=2}^t\left( (\ell_{j+1}-\ell_j)m_j+\frac{m_j^2+m_j}{2}	
		 + \frac{\ell_j(\ell_j+1)}{2}\right) }\cdot (-w)_{r+\ell_2}}
		{(q)^{2k-2}_{\infty}\cdot (q)_{r}(q)_{r+\ell_2}\cdot (q)_{\ell_2}\cdots(q)_{\ell_t}}\\
\end{align*}	
For the second term arising from $-\zeta_s^{-1}$ of $(1-\zeta_s^{-1})$, we similarly get:
\begin{align*}
	&\frac{1}{(q)^{2k-2}_{\infty}}
		\sum_{
			\substack{
				r\geq 0\\
				\forall j,\,\, \ell_j,m_j\geq 0
			}
			}
			\dfrac{-(-1)^{\sum_{j=2}^t(\ell_{j}+m_j)}
			\cdot 
			q^{r+\ell_{s+1}+m_s+1+\sum_{j=2}^t\left( (\ell_{j+1}-\ell_j)m_j+\frac{m_j^2+m_j}{2} 
			 + \frac{\ell_j(\ell_j+1)}{2}\right) }\cdot (-w)_{r+\ell_2}}
			{(q)_{r}(q)_{r+\ell_2}\,\,(q)_{\ell_2}\cdots(q)_{\ell_t}}		
\end{align*}	
Combining the two terms, we get:
\begin{align*}
&
\sum_{
	\substack{
		r\geq 0\\
		\forall j,\,\, \ell_j,m_j\geq 0
	}
	}
	\dfrac{(-1)^{\sum_{j=2}^t(\ell_{j}+m_j)}
	\cdot (1-q^{m_s+1})\cdot 
	q^{r+\ell_{s+1}+\sum_{j=2}^t\left( (\ell_{j+1}-\ell_j)m_j+\frac{m_j^2+m_j}{2} 
	 + \frac{\ell_j(\ell_j+1)}{2}\right) }\cdot (-w)_{r+\ell_2}}
	{(q)_{\infty}^{2k-2}\cdot (q)_{r}(q)_{r+\ell_2}\,\,(q)_{\ell_2}\cdots(q)_{\ell_t}}		
\end{align*}	
At this point, exactly as in \cite{JenMil-double2}, 
we view this as a $x\mapsto 1$ value of:
\begin{align*}
	&
	\sum_{
		\substack{
			r\geq 0\\
			\forall j,\,\, \ell_j,m_j\geq 0
		}
		}
	\dfrac{(-1)^{\sum_{j=2}^t(\ell_{j}+m_j)}
		\cdot (1-q^{m_s+1})\cdot 
		q^{r+\ell_{s+1}+\sum_{j=2}^t\left( (\ell_{j+1}-\ell_j)m_j+\frac{m_j^2+m_j}{2} 
		 + \frac{\ell_j(\ell_j+1)}{2}\right) }\cdot (-xw)_{r+\ell_2}}
		{(q)_{\infty}^{2k-2}\cdot (q)_{r}(xq)_{r+\ell_2}\,\,(q)_{\ell_2}\cdots(q)_{\ell_t}}		
\end{align*}	
Using Heine's theorem \eqref{eqn:Heine} with $a=0$, $b=-xwq^{\ell_2}$, $c=xq^{\ell_2+1}$, $t=q$ to rewrite the inner sum over $r$, that is,
\begin{align*}
	\sum_{r\geq 0}\frac{(-xw)_{r+\ell_2}q^{r}}{(q)_r(xq)_{r+\ell_2}}
	=
	\dfrac{(-xw)_{\ell_2}}{(xq)_{\ell_2}}\sum_{r\geq 0}
	\dfrac{(-xwq^{\ell_2})q^{r}}{(q)_r(xq^{\ell_2+1})_r}
	=
	\frac{(-xw)_{\infty}}{(xq)_{\infty}(q)_\infty}
	\sum_{r\geq 0}(-w^{-1}q)_{r}(-1)^rx^rw^rq^{\ell_2r},
\end{align*}	
we arrive at:
\begin{align*}
	&
	\sum_{
		\substack{
			r\geq 0\\
			\forall j,\,\, \ell_j,m_j\geq 0
		}
		}
	\dfrac{(-xw)_{\infty}x^rw^r	(-1)^{r+\sum_{j=2}^t(\ell_{j}+m_j)}
		\cdot (1-q^{m_s+1})\cdot 
		q^{\ell_2r+\ell_{s+1}+\sum_{j=2}^t\left( (\ell_{j+1}-\ell_j)m_j+\frac{m_j^2+m_j}{2} 
		 + \frac{\ell_j(\ell_j+1)}{2}\right) }\cdot (-w^{-1}q)_{r}}
		{(xq)_{\infty}(q)_{\infty}^{2k-1}\cdot \,\,(q)_{\ell_2}\cdots(q)_{\ell_t}}.		
\end{align*}	
Now we evaluate inner sums on $\ell_j$'s:
\begin{align*}
	\sum_{\ell_2\geq 0}\frac{(-1)^{\ell_2}q^{\ell_2(r-m_2+1)}q^{\frac{\ell_2(\ell_2-1)}{2}}}{(q)_{\ell_2}}
	= (q^{r-m_2+1})_{\infty}=\frac{(q)_{\infty}}{(q)_{r-m_2}}.
\end{align*}
For $j\neq 2, s+1$:
\begin{align*}
	\sum_{\ell_j\geq 0}\frac{(-1)^{\ell_j}q^{\ell_j(1+m_{j-1}-m_j)}q^{\frac{\ell_j(\ell_j-1)}{2}}}{(q)_{\ell_j}}
	=(q^{1+m_{j-1}-m_j})_{\infty}
	=\frac{(q)_{\infty}}{(q)_{m_{j-1}-m_j}}.
\end{align*}
If $j=s+1$ and $s\neq k$:
\begin{align*}
	\sum_{\ell_{s+1}\geq 0}\frac{(-1)^{\ell_{s+1}}q^{\ell_{s+1}(2+m_{s}-m_{s+1})}q^{\frac{\ell_j(\ell_j-1)}{2}}}{(q)_{\ell_{s+1}}}
	=(q^{2+m_{s}-m_{s+1}})_{\infty}
	=\frac{(q)_{\infty}}{(q)_{m_{s}-m_{s+1}+1}}.
\end{align*}
If $s=k$, then by convention, $\ell_{s+1}=\ell_{t+1}=0$, and we may simply ignore that term; all of the inner sums over $\ell_j$ are already accounted for.
We thus reach:
\begin{align*}
	&
	\sum_{
		\substack{
			r\geq 0\\
			\forall j,\,\, m_j\geq 0
		}
		}
	\dfrac{(-xw)_{\infty}x^rw^r	(-1)^{r+\sum_{j=2}^t m_j}
		\cdot (1-q^{m_s+1})\cdot 
		q^{\sum_{j=2}^t \frac{m_j^2+m_j}{2} 
		  }\cdot (-w^{-1}q)_{r}}
		{(xq)_{\infty}(q)_{\infty}^{t}\cdot \,\,(q)_{r-m_2}\,\,
		(q)_{m_2-m_3}\cdots (q)_{m_{s-1}-m_s}(q)_{m_{s}-m_{s+1}+1}(q)_{m_{s+1}-m_{s+2}}\cdots (q)_{m_{t-1}-m_t} }.
\end{align*}	
Note that if $s=k$, then the denominator is simply 
$(xq)_{\infty}(q)_{\infty}^{t}\cdot \,\,(q)_{r-m_2}\,\,
(q)_{m_2-m_3}\cdots (q)_{m_{t-1}-m_t}$.
Now, the inner sum over $r$ is:
\begin{align*}
	\sum_{r\geq 0} \frac{(-1)^rx^rw^r(-w^{-1}q)_r}{(q)_{r-m_2}}
	=(-1)^{m_2}x^{m_2}w^{m_2}(-w^{-1}q)_{m_2}\frac{(xq)_{\infty}}{(xq)_{m_2}(-xw)_{\infty}}.
\end{align*}
Combining everything and letting $x \to 1$ we finally get:
\begin{align*}
	&
	\sum_{
		\substack{
			\forall j,\,\, m_j\geq 0
		}
		}
	\dfrac{(-1)^{\sum_{j=3}^t m_j}\cdot 
		w^{m_2}(-w^{-1}q)_{m_2}
		\cdot (1-q^{m_s+1})\cdot 
		q^{\sum_{j=2}^t \frac{m_j^2+m_j}{2}}}
		{(q)_{\infty}^{t}\cdot \,\,(q)_{m_2}\cdot 
		(q)_{m_2-m_3}\cdots (q)_{m_{s-1}-m_s}(q)_{m_{s}-m_{s+1}+1}(q)_{m_{s+1}-m_{s+2}}\cdots (q)_{m_{t-1}-m_t} }.		
\end{align*}	
Now changing the indices, we have the required sum:
\begin{align*}
	&
	\sum_{
		\substack{
			m_1,\cdots,m_{t-1}\geq 0
		}
		}
	\dfrac{(-1)^{\sum_{j=2}^{t-1} m_j}
		\cdot (1-q^{m_{s-1}+1})\cdot 
		q^{\sum_{j=1}^{t-1} \frac{m_j^2+m_j}{2}}
		\cdot w^{m_1}(-w^{-1}q)_{m_1}}
		{(q)_{\infty}^{t}\cdot \,\,(q)_{m_1}\cdot 
		(q)_{m_1-m_2}\cdots (q)_{m_{s-2}-m_{s-1}}(q)_{m_{s-1}-m_{s}+1}(q)_{m_{s}-m_{s+1}}\cdots (q)_{m_{t-2}-m_{t-1}} }.
\end{align*}	
\end{proof}

The proposition above handles the case when the number of variables, $t$, in $\dn_{t,s}$ is at least $2$.
We have the following for the case $t=1$.

\begin{prop}
	\label{prop:doublepoleBailey1}
	The following holds.
	\begin{align*}
		\sum_{n\geq 0}\frac{(-w)_nq^{an}}{(q)_n^2}
		=\frac{1}{(q)_\infty}\sum_{n\geq 0} \frac{(-1)^nq^{\frac{n^2+n}{2}}(-wq^{a+n})_\infty}{(q)_n(q^{a+n})_\infty}.
	\end{align*}
	In particular, $a=1$ gives $\dn_{1,2}(w,q)$ and $a=2$ gives $\dn_{1,1}(w,q)$.
\end{prop}
\begin{proof}
	We have:
	\begin{align*}
		\sum_{n\geq 0}&\frac{(-w)_nq^{an}}{(q)_n^2}
		=\frac{1}{(q)_\infty}\sum_{n\geq 0}\frac{(-w)_nq^{an}(q^{n+1})_{\infty}}{(q)_n}
		=\frac{1}{(q)_\infty}\sum_{n\geq 0}\frac{(-w)_nq^{an}}{(q)_n}\sum_{j\geq 0}(-1)^j\frac{q^{j(n+1)+\frac{j^2-j}{2}}}{(q)_j}\nonumber\\
		&=\frac{1}{(q)_\infty}\sum_{j\geq 0}(-1)^j\frac{q^{\frac{j^2+j}{2}}}{(q)_j}\sum_{n\geq 0}\frac{(-w)_nq^{(a+j)n}}{(q)_n}
		=\frac{1}{(q)_\infty}\sum_{j\geq 0}(-1)^j\frac{q^{\frac{j^2+j}{2}}(-wq^{a+j})_\infty}{(q)_j(q^{a+j})_\infty},
	\end{align*}
	where in the second equality we have used \eqref{eqn:euler2} and in the last, we have used \eqref{eqn:qbin}.
\end{proof}

\section{Nahm-type sums with double poles via Bailey machinery}
\label{sec:dpnsBailey}

We will start with the Slater's Bailey pair $B(3)$ relative to $a=q$ \cite{Sla-pairs}: 
\begin{align}
	\beta_n=\frac{1}{(q)_n},\,\,\alpha_n=(-1)^nq^{\frac{3n^2+n}{2}}\frac{1-q^{2n+1}}{1-q}.
	\tag{B3}\label{eqn:b3}
\end{align}

The sequence of moves depends on the parity of the number of summation variables (denoted by $t$) used in $\dn_{t,s}$.
Suppose that $t\geq 2$.
We suppose
\begin{align*}
k&=\left\lceil\frac{t}{2}\right\rceil,\quad\quad 0\leq i\leq k.
\end{align*}
We denote:
\begin{align*}
	\lambda&=k-i-1, \quad\quad 
	\mu=t-2-\lambda=\begin{cases}
		k+i-1,& t=2k\\
		k+i-2,& t=2k-1.
	\end{cases}
\end{align*}

Now, 
\begin{enumerate}
\item If $i=k$, use the move \eqref{move:Fb} $t-2$ times, followed by \eqref{move:Fwb} once.
\item If $0\leq i\leq k-1$, we apply \eqref{move:Fb} $\lambda$ times, followed by \eqref{move:Sb} once, 
	followed by \eqref{move:Fb} $\mu$ times, and finally \eqref{move:Fwb} once.
\end{enumerate}
If $t$ is even (respectively, odd), we denote the final $\beta$ 
thus obtained by $\beta_n^{(k,i,0)}(w,q)$ (respectively, $\beta_n^{(k,i,1)}(w,q)$).

Using the explicit description of these moves along with Theorem 
\ref{prop:doublepoleBailey},
it can be seen without much effort that:
\begin{align}
	\beta_{\infty}^{(k,i,0)}(w,q)&=\lim_{n\rightarrow\infty}\beta_n^
	{(k,i,0)}(w,q)
	=\frac{(q)_{\infty}^{2k-1}}{(-wq)_\infty}\dn_{2k,k+i+1}(w,q), 
	\label{eqn:betakiinf_e}\\
	\beta_{\infty}^{(k,i,1)}(w,q)&=\lim_{n\rightarrow\infty}\beta_n^
	{(k,i,1)}(w,q)
	=\frac{(q)_{\infty}^{2k-2}}{(-wq)_\infty}\dn_{2k-1,k+i}(w,q).
	\label{eqn:betakiinf_o}
\end{align}

Now we find formulas for the corresponding $\alpha_n^{(k,i,p)}(w,q)$ for $p=0,1$.

For $i=k$, we have for all $n\geq 0$:
\begin{align}
	\alpha_n^{(k,k,0)}(w,q)=
	(-1)^n q^{(k+1)n^2+kn}\frac{1-q^{2n+1}}{1-q}\frac{w^n(-w^{-1}q)_n}{(-wq)_n},
	\label{eqn:alphakk_e}\\
	\alpha_n^{(k,k,1)}(w,q)=
	q^{(k+\frac{1}{2})n^2+(k-\frac{1}{2})n}\frac{1-q^{2n+1}}{1-q}\frac{w^n(-w^{-1}q)_n}{(-wq)_n}.
	\label{eqn:alphakk_o}
\end{align}

For $i\leq k-1$, one can directly see after a straight-forward calculation that for both $p=0,1$:
\begin{align}
	\alpha_0^{(k,i,p)}(w,q)&=\frac{1+(-1)^{k-i}q^{k-i+1}}{1+q}=1-q+q^2-\cdots +(-1)^{k-i}q^{k-i}.
	\label{eqn:GAalphaki0}
\end{align}

Now let $n>0$ and $p=0,1$.
We get:
\begin{align}
	\alpha_n^{(k,i,p)}&(w,q)=\frac{(-1)^{\mu n}q^{(\mu+1)\frac{n^2+n}{2}}w^n(-w^{-1}q)_n}{(1-q)(-wq)_n}
	\left( 
		\frac{q^{2n-1}(1-q^n)}{1+q^n}\cdot(-1)^{(\lambda+1)(n-1)}q^{\frac{(\lambda+3)(n-1)^2+(\lambda+1)(n-1)}{2}}
	\right.\nonumber\\
	&\,\, +\frac{2q^{n}(1-q^{2n+1})}{(1+q^n)(1+q^{n+1})}\cdot(-1)^{(\lambda+1)n}q^{\frac{(\lambda+3)n^2+(\lambda+1)n}{2}}
	 +
	\left.
		\frac{1-q^{n+1}}{1+q^{n+1}}\cdot(-1)^{(\lambda+1)(n+1)}q^{\frac{(\lambda+3)(n+1)^2+(\lambda+1)(n+1)}{2}}
	\right)\nonumber\\
	&=
	\frac{(-1)^{(\lambda+\mu+1)n+(\lambda-1)}q^{\frac{(\lambda+\mu+4)n^2+(\mu-\lambda)n}{2}}w^n(-w^{-1}q)_n}{(1-q)(-wq)_n}
	\left(
		\frac{1-q^n}{1+q^n} 
	\right.\nonumber\\
	& \qquad\qquad\left.
		+ \frac{2(1-q^{n+1})}{(1+q^n)(1+q^{n+1})}(-1)^{\lambda+1}q^{(\lambda+2)n}
		+ \frac{1-q^{n+1}}{1+q^{n+1}}q^{(2\lambda+4)n + (\lambda+2)}
		\right)\nonumber\\
	&=
	\frac{(-1)^{(t+1)n+(k-i)}q^{\frac{(t+2)n^2+(\mu-\lambda)n}{2}}w^n(-w^{-1}q)_n}{(1-q)(-wq)_n}
	\left(
		1-\frac{2q^n}{1+q^n} 
	\right.\nonumber\\
	& \qquad\qquad\left.
		+ \frac{2(-1)^{\lambda+1}q^{(\lambda+2)n}}{1+q^n}
		+ \frac{2(-1)^{\lambda}q^{(\lambda+3)n+1}}{1+q^{n+1}}
		- \frac{2q^{(2\lambda+5)n + (\lambda+3)}}{1+q^{n+1}}
		+ q^{(2\lambda+4)n + (\lambda+2)}
	\right)\nonumber\\
	&=\frac{(-1)^{(t+1)n+(k-i)}q^{\frac{(t+2)n^2+(\mu-\lambda-2i)n}{2}}w^n(-w^{-1}q)_n}{(1-q)(-wq)_n}
	\left(
		q^{in}-2q^{(i+1)n}+2q^{(i+2)n}-\cdots+2(-1)^{k-i}q^{kn}
	\right.\nonumber\\
	&\qquad\qquad \left.
		- 2(-1)^{k-i}q^{(k+2)n+1}+2(-1)^{k-i}q^{(k+3)n+2}
		-\cdots +2q^{(2k-i+1)n+k-i}-q^{(2k-i+2)n+(k-i+1)}
	\right),\label{eqn:GAalphaki}
\end{align}
where the last equality follows from easy geometric sum formulas analogous to \eqref{eqn:GAalphaki0}.

Note that $\mu-\lambda-2i=0$ if $t=2k$ (i.e., $p=0$) and $\mu-\lambda-2i=-1$ if $t=2k-1$ (i.e., $p=1$).
It is not hard to check that \eqref{eqn:GAalphaki} with $n=0$ exactly gives \eqref{eqn:GAalphaki0}.
Similarly, for $i=k$, we understand the term in parentheses of \eqref{eqn:GAalphaki} as
$q^{kn}-q^{(k+2)n+1}$, and with this, \eqref{eqn:GAalphaki} reproduces \eqref{eqn:alphakk_e} (or \eqref{eqn:alphakk_o}).
We may thus use \eqref{eqn:GAalphaki} for all $n\geq0$ and all $0\leq i \leq k$ uniformly.

Using the equation that asserts that $\alpha^{(k,i,p)}$ and $\beta^{(k,i,p)}$ indeed form a Bailey pair relative to $a=q$,
letting $n\rightarrow\infty$ in this equation and using \eqref{eqn:betakiinf_e}, \eqref{eqn:betakiinf_o}, we deduce:

\begin{thm}\label{thm:dp-fermionic}
	If $t=2k$ ($k\geq 1$) and $0\leq i\leq k$, we have:
	\begin{align*}
		&\frac{(q)_{\infty}^{2k-1}}{(-wq)_\infty}\dn_{2k,k+i+1}(w,q)\nonumber\\
		&=\dfrac{1}{(q)^2_{\infty}}
		\sum_{r\geq 0}
		\frac{(-1)^{r+(k-i)}q^{(k+1)r^2}w^r(-w^{-1}q)_r}{(-wq)_r}
	\left(
		q^{ir}-2q^{(i+1)r}+2q^{(i+2)r}-\cdots+2(-1)^{k-i}q^{kr}
	\right.\nonumber\\
	&\qquad\qquad \left.
		- 2(-1)^{k-i}q^{(k+2)r+1}+2(-1)^{k-i}q^{(k+3)r+2}
		-\cdots +2q^{(2k-i+1)r+k-i}-q^{(2k-i+2)r+(k-i+1)}
	\right).
	\end{align*}
	If $t=2k-1$ ($k\geq 2$) and $0\leq i\leq k$, we have:
	\begin{align*}
		&\frac{(q)_{\infty}^{2k-2}}{(-wq)_\infty}\dn_{2k-1,k+i}(w,q)\nonumber\\
		&=\dfrac{1}{(q)^2_{\infty}}
		\sum_{r\geq 0}
		\frac{(-1)^{(k-i)}q^{(k+\frac{1}{2})r^2-\frac{1}{2}r}w^r(-w^{-1}q)_r}{(-wq)_r}
	\left(
		q^{ir}-2q^{(i+1)r}+2q^{(i+2)r}-\cdots+2(-1)^{k-i}q^{kr}
	\right.\nonumber\\
	&\qquad\qquad \left.
		- 2(-1)^{k-i}q^{(k+2)r+1}+2(-1)^{k-i}q^{(k+3)r+2}
		-\cdots +2q^{(2k-i+1)r+k-i}-q^{(2k-i+2)r+(k-i+1)}
	\right).
	\end{align*}
\end{thm}	
We note that the $i=k$ cases of the two identities above were established in \cite{JenMil-double2}.

\section{Sum=Product identities}
\label{sec:sum=product}
We now deduce various ``sum=product'' identities with an even number of summation variables in $\dn_{t,s}$, i.e., with $t=2k$.
\subsection{Andrews-Gordon series with $w\rightarrow 0$} 
We now consider the case $w\rightarrow 0$ and deduce double-pole representations of the 
Gordon-Andrews (odd modulus) series. This generalizes the $i=k$ case established in \cite[Section 5]{JenMil-double1}.

\begin{thm}\label{thm:doublepoleGA}
	For $k\geq 1$, $0\leq i\leq k$, we have the following
	\begin{align*}
	 (-1)^{k-i}\dn_{2k,k+i+1}(0,q) + 2\sum_{j=i+1}^k (-1)^{k-j}\dn_{2k,k+j+1}(0,q) 
		&
		=\frac{(q^{k-i+1}, q^{k+i+2}, q^{2k+3}\,\,;\,\,q^{2k+3})_{\infty}}{(q)^{2k+1}_{\infty}}.
	\end{align*}
\end{thm}
\begin{proof} 
	We begin by analyzing $(-1)^{k-i}\alpha_{n}^{(k,i,0)}(0,q) + 2\sum_{j=i+1}^k (-1)^{k-j}\alpha_{n}^{(k,j,0)}(0,q)$.

	First we consider a fixed $n>0$.
	Taking the limit $w\rightarrow 0$ in \eqref{eqn:GAalphaki},
	the outer factor becomes:
	\begin{align}
		\frac{(-1)^{n+(k-i)} q^{\frac{(2k+3)n^2+n}{2}} }{(1-q)}.
		\label{eqn:AGout}
	\end{align}
	We tentatively keep this outer factor aside, remembering that 
	it depends solely on $k$ and $n$ (which we have fixed).
	Using again the formula \eqref{eqn:GAalphaki} for $(-1)^{k-i}\alpha^{(k,i,0)}_n(0,q)$, we observe that there are
	two strings of monomials -- one in which powers of $q$ advance by $n$ and the other where they advance by $n+1$.
	When we consider $(-1)^{k-i}\alpha_{n}^{(k,i,0)}(0,q) + 2\sum_{j=i+1}^k (-1)^{k-j}\alpha_{n}^{(k,j,0)}(0,q)$, the strings corresponding to powers of $q^n$ can be arranged in the following way:
	\begin{align}
		\begin{array}{lccccl}
			q^{in} & -2q^{(i+1)n}	&+2q^{(i+2)n}	&-2q^{(i+2)n}&\cdots&+2(-1)^{k-i}q^{kn}\\
			       &  + 2q^{(i+1)n}	&-4q^{(i+2)n}	&+4q^{(i+2)n}&\cdots&+4(-1)^{k-i-1}q^{kn}\\
				   &  				& + 2q^{(i+2)n}	&-4q^{(i+2)n}&\cdots&+4(-1)^{k-i-2}q^{kn}\\
				   &  				&				& + 2q^{(i+2)n}&\cdots&+4(-1)^{k-i-3}q^{kn}\\
				   &				&				&			 & \ddots & \vdots \\
				   &				&				&			 &  & +2q^{kn} \\
		\end{array}			
		\label{eqn:trianglecancel}
	\end{align}
	These terms add up to $q^{in}$ as can be seen from the fact that all the column sums except for the first column are $0$.
	Strings that advance by $q^{n+1}$ also lead to a similar arrangement:
	\begin{align*}
		\begin{array}{rccccl}
			-2(-1)^{k-i}q^{(k+2)n+1} & \cdots & +2q^{(2k-i-1)n+(k-i-2)} &-2q^{(2k-i)n+(k-i-1)} & +2q^{(2k-i+1)n+(k-i)} & { -q^{(2k-i+2)n+(k-i+1)}}\\
			-4(-1)^{k-i-1}q^{(k+2)n+1} & \cdots & -4q^{(2k-i-1)n+(k-i-2)} &+4q^{(2k-i)n+(k-i-1)} & -2q^{(2k-i+1)n+(k-i)} & \\
			-4(-1)^{k-i-2}q^{(k+2)n+1} & \cdots & +4q^{(2k-i-1)n+(k-i-2)} &-2q^{(2k-i)n+(k-i-1)} &  & \\
			-4(-1)^{k-i-3}q^{(k+2)n+1} & \cdots & -2q^{(2k-i-1)n+(k-i-2)} & &  & \\
			\vdots & \iddots &  & &  & \\
			-2q^{(k+2)n+1} & &  & &  &.
		\end{array}			
	\end{align*}
	Again, all the column sums except the last are $0$, and so these terms add up to $-q^{(2k-i+2)n+(k-i+1)}$.

	We thus conclude that for $n>0$, we have
	\begin{align*}
		(-1)^{k-i}\alpha_{n}^{(k,i,0)}(0,q) &+ 2\sum_{j=i+1}^k (-1)^{k-j}\alpha_{n}^{(k,j,0)}(0,q)=
		\frac{(-1)^{n}q^{\frac{(2k+3)n^2+n}{2}}}{1-q}\left(q^{in}-q^{(2k-i+2)n+(k-i+1)}\right)
		\nonumber\\
		&=\frac{(-1)^{n}q^{\frac{(2k+3)n^2+(2i+1)n}{2}}}{1-q}\left(1-q^{(k-i+1)(2n+1)}\right).
	\end{align*}

	Now we consider the case $n=0$. 
	Using \eqref{eqn:GAalphaki0} and a similar arrangement of terms as in \eqref{eqn:trianglecancel}, we see that 
	\begin{align*}
	 	(-1)^{k-i}\alpha_{0}^{(k,i,0)}(0,q) + 2\sum_{j=i+1}^k (-1)^{k-j}\alpha_{0}^{(k,j,0)}(0,q)
	 	=1+q+\cdots +q^{k-i} = \frac{1-q^{k-i+1}}{1-q}.
	\end{align*}

	Putting everything together, recalling that $a=q$, we now see that:
	\begin{align*}
		(-1)^{k-i}&\beta_{\infty}^{(k,i,0)}(0,q) + 2\sum_{j=i+1}^k (-1)^{k-j}\beta_{\infty}^{(k,j,0)}(0,q) \nonumber\\
		&=\dfrac{1}{(q)_{\infty}(q^2\,\,;\,\,q)_{\infty}}
		\sum_{n=0}^{\infty}\left((-1)^{k-i}\alpha_{n}^{(k,i,0)}(0,q) + 2\sum_{j=i+1}^k (-1)^{k-j}\alpha_{n}^{(k,j,0)}(0,q)\right) \nonumber\\
		&=\dfrac{1}{(q)_{\infty}(q^2\,\,;\,\,q)_{\infty}}
		\sum_{n=0}^{\infty} \frac{(-1)^{n}q^{\frac{(2k+3)n^2+(2i+1)n}{2}}}{1-q}\left(1-q^{(k-i+1)(2n+1)}\right)\nonumber\\
		&= \dfrac{(q^{k-i+1}, q^{k+i+2}, q^{2k+3}\,\,;\,\,q^{2k+3})_{\infty}}{(q)_{\infty}^2}.
	\end{align*}
	Where, in the very last step, we have used the Jacobi triple-product identity.
	Now the required statement follows, using \eqref{eqn:betakiinf_e}.
\end{proof}	

\subsection{Andrews-Bressoud series with $w \to 1$}

We now consider the case $w\rightarrow 1$ and deduce double-pole representations of the 
Andrews-Bressoud (even modulus) series.

\begin{thm}\label{thm:doublepoleAB}
	For $k\geq 1$, $0\leq i\leq k$, we have the following:
	\begin{align*}
		(-1)^{k-i}\dn_{2k,k+i+1}(1,q) + 2\sum_{j=i+1}^k (-1)^{k-j}\dn_{2k,k+j+1}(1,q) 
		=
		\dfrac{(-q;\,\,q)_{\infty}(q^{k-i+1}, q^{k+i+1}, q^{2k+2}\,\,;\,\,q^{2k+2})_{\infty}}{(q)_{\infty}^{2k+1}}.
	\end{align*}
\end{thm}
\begin{proof}
	The $i=k$ case was handled in \cite[Section 5]{JenMil-double2}.
	The proof here is exactly analogous to the proof of Theorem \ref{thm:doublepoleGA}.

	For a fixed $n>0$,  the outer factor of \eqref{eqn:GAalphaki} for $\alpha_n^{(k,i,0)}(1,q)$ becomes:
	\begin{align*}
		\frac{(-1)^{n+(k-i)} q^{(k+1)n^2} }{(1-q)}.
	\end{align*}
	Rest of the analysis being exactly the same as in the proof of Theorem \ref{thm:doublepoleGA}, we see that for $n>0$:
	\begin{align*}
		(-1)^{k-i}\alpha_{n}^{(k,i,0)}(1,q) &+ 2\sum_{j=i+1}^k (-1)^{k-j}\alpha_{n}^{(k,j,0)}(1,q)=
		\frac{(-1)^{n}q^{(k+1)n^2}}{1-q}\left(q^{in}-q^{(2k-i+2)n+(k-i+1)}\right)
		\nonumber\\
		&=\frac{(-1)^{n}q^{(k+1)n^2+in}}{1-q}\left(1-q^{(k-i+1)(2n+1)}\right).
	\end{align*}
	The formula for $\alpha_0^{(k,i,0)}(w,q)$ being independent of $w$, we again have:
	\begin{align*}
		(-1)^{k-i}\alpha_{0}^{(k,i,0)}(1,q) + 2\sum_{j=i+1}^k (-1)^{k-j}\alpha_{0}^{(k,j,0)}(1,q)
		=1+q+\cdots +q^{k-i} = \frac{1-q^{k-i+1}}{1-q}.
   \end{align*}
   Combining, we get:
   \begin{align*}
		(-1)^{k-i}&\beta_{\infty}^{(k,i,0)}(1,q) + 2\sum_{j=i+1}^k (-1)^{k-j}\beta_{\infty}^{(k,j,0)}(1,q) \nonumber\\
		&=\dfrac{1}{(q)_{\infty}(q^2\,\,;\,\,q)_{\infty}}
		\sum_{n=0}^{\infty}\left((-1)^{k-i}\alpha_{n}^{(k,i,0)}(1,q) + 2\sum_{j=i+1}^k (-1)^{k-j}\alpha_{n}^{(k,j,0)}(1,q)\right) \nonumber\\
		&=\dfrac{1}{(q)_{\infty}(q^2\,\,;\,\,q)_{\infty}}
		\sum_{n=0}^{\infty} \frac{(-1)^{n}q^{(k+1)n^2+in}}{1-q}\left(1-q^{(k-i+1)(2n+1)}\right)\nonumber\\
		&= \dfrac{(q^{k-i+1}, q^{k+i+1}, q^{2k+2}\,\,;\,\,q^{2k+2})_{\infty}}{(q)_{\infty}^2}.
	\end{align*}
	Where, in the very last step, we have used the Jacobi triple-product identity.
	Now the required statement follows, using \eqref{eqn:betakiinf_e}.
\end{proof}	

\subsection{Andrews-Bressoud series with $w \to q^{1/2}$}

We now consider the case $w\rightarrow q^{1/2}$ and deduce double-pole representations for \emph{some} of the
Andrews-Bressoud (even modulus) series.

\begin{thm}\label{thm:doublepoleAB2}
	For $k\geq 1$, $0\leq i\leq k$, we have the following:
	\begin{align*}
		(-1)^{k-i}&\dn_{2k,k+i+1}(q^{1/2},q) + 2\sum_{j=i+1}^k (-1)^{k-j}\dn_{2k,k+j+1}(q^{1/2},q) 
		\nonumber\\
		&\quad +q^{1/2}\left((-1)^{k-i-1}\dn_{2k,k+i+2}(q^{1/2},q) + 2\sum_{j=i+2}^k (-1)^{k-j}\dn_{2k,k+j+1}(q^{1/2},q)\right) 
		\nonumber\\
		&=\dfrac{(-q^{1/2},\,\,q)_{\infty}(q^{k-i+\frac{1}{2}}, q^{k+i+\frac{3}{2}}, q^{2k+2}\,\,;\,\,q^{2k+2})_{\infty}}{(q)^{2k+1}_{\infty}}.
	\end{align*}
Note that if $i=k$ then we do not have the terms multiplied with $q^{1/2}$ in the left-hand side.
\end{thm}
\begin{proof}
	For a fixed $n>0$,  the outer factor in $\alpha_n^{(k,i,0)}(q^{1/2},q)$ becomes:
	\begin{align*}
		\frac{(-1)^{n+(k-i)} q^{(k+1)n^2+\frac{n}{2}}(1+q^{\frac{1}{2}}) }{(1-q)(1+q^{n+\frac{1}{2}})}.
	\end{align*}
	Rest of the analysis being exactly the same as in the proof of Theorem \ref{thm:doublepoleGA}, we see that for $n>0$:
	\begin{align*}
		(-1)^{k-i}&\alpha_{n}^{(k,i,0)}(q^{1/2},q) + 2\sum_{j=i+1}^k (-1)^{k-j}\alpha_{n}^{(k,j,0)}(q^{1/2},q)
		\nonumber\\
		&\qquad+q^{1/2}\left((-1)^{k-i-1}\alpha_{n}^{(k,i+1)}(q^{1/2},q) + 2\sum_{j=i+2}^k (-1)^{k-j}\alpha_{n}^{(k,j,0)}(q^{1/2},q)\right)
		\nonumber\\
		&=\frac{(-1)^{n}q^{(k+1)n^2+\frac{n}{2}}(1+q^{\frac{1}{2}})}{(1-q)(1+q^{n+\frac{1}{2}})}
		\left(q^{in}-q^{(2k-i+2)n+(k-i+1)}+q^{1/2}\left(q^{(i+1)n}-q^{(2k-i+1)n+(k-i)}\right)\right)
		\nonumber\\
		&=\frac{(-1)^{n}q^{(k+1)n^2+\frac{n}{2}}(1+q^{\frac{1}{2}})}{(1-q)(1+q^{n+\frac{1}{2}})}
		(q^{in} - q^{(2k-i+1)n+(k-i+\frac{1}{2})})(1+q^{n+\frac{1}{2}})
		\nonumber\\
		&=\frac{(-1)^{n}q^{(k+1)n^2+\frac{n}{2}}}{(1-q^{\frac{1}{2}})}
		(q^{in} - q^{(2k-i+1)n+(k-i+\frac{1}{2})})
		\nonumber\\
		&=\frac{(-1)^{n}q^{(k+1)n^2+(i+\frac{1}{2})n}}{(1-q^{\frac{1}{2}})}
		(1 - q^{(k-i+\frac{1}{2})(2n+1)}).
	\end{align*}
	Similarly, we see:
	\begin{align*}
		(-1)^{k-i}&\alpha_{0}^{(k,i,0)}(q^{1/2},q) + 2\sum_{j=i+1}^k (-1)^{k-j}\alpha_{0}^{(k,j,0)}(q^{1/2},q)
		\nonumber\\
		&\qquad+q^{1/2}\left((-1)^{k-i-1}\alpha_{0}^{(k,i+1)}(q^{1/2},q) + 2\sum_{j=i+2}^k (-1)^{k-j}\alpha_{0}^{(k,j,0)}(q^{1/2},q)\right)
		\nonumber\\
		&=\frac{1-q^{k-i+1}}{1-q}+q^{\frac{1}{2}}\frac{1-q^{k-i}}{1-q}
		=\frac{(1+q^{\frac{1}{2}})(1-q^{k-i+\frac{1}{2}})}{1-q}
		=\frac{1-q^{k-i+\frac{1}{2}}}{1-q^{\frac{1}{2}}}.
	\end{align*}
	Combining, we see:
	\begin{align*}
		(-1)^{k-i}&\beta_{\infty}^{(k,i,0)}(q^{1/2},q) + 2\sum_{j=i+1}^k (-1)^{k-j}\beta_{\infty}^{(k,j,0)}(q^{1/2},q)
		\nonumber\\
		&\qquad +q^{1/2}\left.\left((-1)^{k-i-1}\beta_{\infty}^{(k,i+1)}(q^{1/2},q) + 2\sum_{j=i+2}^k (-1)^{k-j}\beta_{\infty}^{(k,j,0)}(q^{1/2},q)\right)\right)
		\nonumber\\
		&=\frac{1}{(q)_{\infty}(q^2;\,\,q)_{\infty}}
		\sum_{n=0}^\infty\left(
		(-1)^{k-i}\alpha_{n}^{(k,i,0)}(q^{1/2},q) + 2\sum_{j=i+1}^k (-1)^{k-j}\alpha_{n}^{(k,j,0)}(q^{1/2},q)\right.
		\nonumber\\
		&\qquad +q^{1/2}\left.\left((-1)^{k-i-1}\alpha_{n}^{(k,i+1)}(q^{1/2},q) + 2\sum_{j=i+2}^k (-1)^{k-j}\alpha_{n}^{(k,j,0)}(q^{1/2},q)\right)\right)
		\nonumber\\			
		&=\frac{1}{(q)_{\infty}(q^{2};\,\,q)_{\infty}(1-q^{\frac{1}{2}})}
		\sum_{n=0}^{\infty}
		\left((-1)^{n}q^{(k+1)n^2+(i+\frac{1}{2})n}
		(1-q^{(k-i+\frac{1}{2})(2n+1)})\right)
		\nonumber\\
		&=\dfrac{(q^{k-i+\frac{1}{2}}, q^{k+i+\frac{3}{2}}, q^{2k+2}\,\,;\,\,q^{2k+2})_{\infty}}{(q)_{\infty}(q^{2};\,\,q)_{\infty}(1-q^{\frac{1}{2}})}.
	\end{align*}
	Now the required statement follows, using \eqref{eqn:betakiinf_e}.
Note that after we let $q\mapsto q^2$, we get proper identities modulo $4k+4$.
\end{proof}

\section{Identities for Rogers' false theta functions}
\label{sec:ft}

Here we give identities involving double pole sums $\dn_{t,s}$
with an odd number of summation variables $t=2k-1$. 
Now, instead of ``sum=product'' identities, we get identities involving
Rogers' false theta function.
All the proofs are similar to the ones above.

\subsection{Identities with $w\rightarrow 0$}
We start the case $w= 0$ and deduce double-pole representations of all unary false theta functions, thus generalizing \cite[Theorem 5.1]{JenMil-double1}.

\begin{thm} \label{thm:doublepoleFalseAG}
	For $k \geq 1$, $0 \leq i \leq k$, we have the following identities:
	\begin{align*}
	 (-1)^{k-i}\dn_{2k-1,k+i}(0,q) + 2\sum_{j=i+1}^k (-1)^{k-j}\dn_{2k-1,k+j}(0,q) 
		&
		=\frac{1}{(q)_\infty^{2k}} \sum_{n \in \mathbb{Z}} \sgn^*(n) q^{ (k+1)n^2+in}.
	\end{align*}
	Note that for $i=0$, the right-hand side reduces to $\dfrac{1}{(q)_\infty^{2k}}$.
\end{thm}
\begin{proof} 
	For $k\geq 2$, Theorem \ref{thm:dp-fermionic} applies and
	the proof is similar to the proof of 
	Theorem \ref{thm:doublepoleGA}, with one of significant changes being
	that the outer factor analogous to \eqref{eqn:AGout} does not have the $(-1)^n$ part.
	We omit rest of the details.

	For $k=1$, we use Proposition \ref{prop:doublepoleBailey1} to see:
	\begin{align*}
	\dn_{1,2}(0,q) &= \dfrac{1}{(q)_\infty^2}\sum_{j\geq 0}(-1)^j q^{\frac{j^2+j}{2}}
	=\dfrac{1}{(q)_\infty^2}\sum_{n\in\ZZ}\sgn^*(n)q^{2n^2+n}\\
	-\dn_{1,1}(0,q)+2\dn_{1,2}(0,q) &= 
	\dfrac{1}{(q)_\infty^2}\sum_{j\geq 0}(-1)^j q^{\frac{j^2+j}{2}}(1+q^{j+1})
	=\dfrac{1}{(q)_\infty^2},
	\end{align*}
	as required.
\end{proof}
\begin{rem} The above result gives new $q$-series representations of distinguished characters of irreducible modules of the $(1,k+1)$-singlet vertex algebra \cite{BM}. 
\end{rem}

\begin{thm}\label{thm:doublepoleFalseAB}
	For $k\geq 1$, $0\leq i\leq k$, we have the following: 	\begin{align*}
		(-1)^{k-i}\dn_{2k-1,k+i}(1,q) + 2\sum_{j=i+1}^k (-1)^{k-j}\dn_{2k-1,k+j}(1,q) 
		=
		\dfrac{(-q;\,\,q)_{\infty}}{(q)_{\infty}^{2k}}
		\sum_{n\in\ZZ}\sgn^*(n)q^{(k+\frac{1}{2})n^2 + (i-\frac{1}{2})n}.
	\end{align*}
\end{thm}
\begin{proof}
	Again, the proof for $k\geq 2$ is similar to the proof of Theorem \ref{thm:doublepoleAB}.
	For $k=1$, using Proposition \ref{prop:doublepoleBailey1} with $w=1$, we see:
	\begin{align}
		\dn_{1,2}(1,q)&=\frac{(-q;q)_\infty}{(q)_\infty^2}\sum_{j\geq 0}(-1)^j\frac{q^{\frac{j^2+j}{2}}}{(-q)_j}
		=\frac{(-q;q)_\infty}{(q)_\infty^2}\sum_{n\in\ZZ}\sgn^*(n)q^{\frac{3n^2+n}{2}}\label{eqn:ftRogers}\\
		-\dn_{1,2}(1,q)+2\dn_{1,2}(1,q)&=
		\frac{(-q;q)_\infty}{(q)_\infty^2}\sum_{j\geq 0}(-1)^j{q^{\frac{j^2+j}{2}}}
		\left(\frac{2}{(-q)_j}-\frac{1-q^{j+1}}{(-q)_{j+1}} \right)\nonumber\\
		&=\frac{(-q;q)_\infty}{(q)_\infty^2}\sum_{j\geq 0}(-1)^j
		{q^{\frac{j^2+j}{2}}}
		\left(
		\frac{1}{(-q)_j}+2\frac{q^{j+1}}{(-q)_{j+1}}
		\right) 
		\nonumber\\
		&=\frac{(-q;q)_\infty}{(q)_\infty^2}
		\left(\sum_{j\geq 0}(-1)^j\frac{q^{\frac{j^2+j}{2}}}{(-q)_{j}}
		+2\sum_{j\geq 0}(-1)^j\frac{q^{\frac{j^2+j}{2}+j+1}}{(-q)_{j+1}}
		\right)\nonumber\\
		&=\frac{(-q;q)_\infty}{(q)_\infty^2}
		\left(\sum_{j\geq 0}(-1)^j\frac{q^{\frac{j^2+j}{2}}}{(-q)_{j}}
		-2\sum_{j\geq 1}(-1)^j\frac{q^{\frac{j^2+j}{2}}}{(-q)_{j}}
		\right)\nonumber\\
		&=\frac{(-q;q)_\infty}{(q)_\infty^2}
		\left(1-\sum_{j\geq 1}(-1)^j\frac{q^{\frac{j^2+j}{2}}}{(-q)_{j}}\right)
		\nonumber\\
		&=\frac{(-q;q)_\infty}{(q)_\infty^2}\sum_{n\in\ZZ}\sgn^*(n)q^{\frac{3n^2-n}{2}}.\nonumber
	\end{align}
	Right-hand side of \eqref{eqn:ftRogers} is well-known due to Rogers \cite{Rog} (see also \cite{Sil-book}).
\end{proof}

\begin{thm}\label{thm:doublepoleFalseAB2}
	For $k\geq 1$, $0\leq i\leq k$, we have the following:
	\begin{align*}
		(-1)^{k-i}&\dn_{2k-1,k+i}(q^{\frac12},q) + 2\sum_{j=i+1}^k (-1)^{k-j}\dn_{2k-1,k+j}(q^{\frac12},q) 
		\nonumber\\
		&\quad +q^{\frac12}\left((-1)^{k-i-1}\dn_{2k-1,k+i+1}(q^{\frac12},q) + 2\sum_{j=i+2}^k (-1)^{k-j}\dn_{2k-1,k+j}(q^{\frac12},q)\right) 
		\nonumber\\
		&=
		\dfrac{(-q^{\frac12};\,\,q)_{\infty}}{(q)_{\infty}^{2k}}
		\sum_{n\in\ZZ}\sgn^*(n)q^{(k+\frac{1}{2})n^2 + in}.
	\end{align*}
\end{thm}
\begin{proof}
	For $k=1$, using Proposition \ref{prop:doublepoleBailey1} with $w=q^{1/2}$, we see:
	\begin{align}
		\dn_{1,2}(1,q)&=\frac{(-q^{\frac12};q)_\infty}{(q)_\infty^2}\sum_{n\geq 0}(-1)^n\frac{q^{\frac{n^2+n}{2}}}{(-q^{\frac12})_{n+1}}
		=\frac{(-q^{\frac12};q)_\infty}{(q)_\infty^2}
		\sum_{n\in\ZZ}\sgn^*(n)q^{\frac{3n^2}{2}+n}\label{eqn:ftRogers2}\\
		-\dn_{1,1}+(2+q^{\frac12})\dn_{1,2}&=
		\frac{(-q^{1/2};q)_\infty}{(q)_\infty^2}\sum_{n\geq 0}(-1)^n{q^{\frac{n^2+n}{2}}}
		\left( \dfrac{2+q^{\frac12}}{(-q^{\frac12})_{n+1}}-\dfrac{1-q^{n+1}}{(-q^{\frac12})_{n+2}}\right)\nonumber\\
		&=\frac{(-q^{\frac12};q)_\infty}{(q)_\infty^2}\sum_{n\geq 0}(-1)^n\dfrac{{q^{\frac{n^2+n}{2}}}(1+q^{\frac12})}{(-q^{1/2})_{n+2}}
		\left(1+q^{n+\frac32}+q^{n+1}\right)\nonumber\\
		&=\frac{(-q^{\frac12};q)_\infty(1+q^{\frac12})}{(q)_\infty^2}
		\left(\sum_{n\geq 0}(-1)^n\dfrac{{q^{\frac{n^2+n}{2}}}}{(-q^{\frac12})_{n+1}}
		+\sum_{n\geq 0}(-1)^n\dfrac{{q^{\frac{n^2+n}{2}+n+1}}}{(-q^{\frac12})_{n+2}}\right)\nonumber\\
		&=\frac{(-q^{\frac12};q)_\infty(1+q^{\frac12})}{(q)_\infty^2}
		\left(\sum_{n\geq 0}(-1)^n\dfrac{{q^{\frac{n^2+n}{2}}}}{(-q^{\frac12})_{n+1}}
		-\sum_{n\geq 1}(-1)^n\dfrac{{q^{\frac{n^2+n}{2}}}}{(-q^{\frac12})_{n+1}}\right)\nonumber\\
		&=\frac{(-q^{\frac12};q)_\infty}{(q)_\infty^2}.\nonumber
	\end{align}
	Here, again, \eqref{eqn:ftRogers2} is due to Rogers \cite{Rog} (see also \cite{Sil-book}).
\end{proof}	

\begin{rem} Theorems  \ref{thm:doublepoleFalseAB2} and \ref{thm:doublepoleFalseAB} give new $q$-series representations of distinguished irreducible characters for the  $(1,2k+1)$ $N=1$ singlet vertex superalgebra in Neveu-Schwarz and Ramond sector, respectively.
\end{rem}

\section{An alternative approach to double pole identities}
\label{sec:alt}

In this part we present a different approach to double pole identities based  on $q$-hypergeometric summations
combined with  Andrews-Gordon identities (Theorem \ref{AG-id}) and $q$-series identities for 
false theta functions (Theorem \ref{false-id}). We also employ the relevant $q$-difference equations (\ref{AG-q-diff}) and (\ref{false-q-diff}).

Let us start with an identity (here $m \geq 0$):
\begin{equation} \label{false}
\sum_{n \geq 0} \frac{q^{(m+1)n}}{(q)_{n}^2}=\frac{1}{(q)_\infty} \sum_{n \geq 0} \frac{(q)_{m} q^{n^2+(m+1)n}}{(q)_n^2},
\end{equation}
an easy consequence of (\ref{Heine}) with $a,b \to 0$, $c\rightarrow q$ and $z=q^{m+1}$. This expression is essentially a difference of two partial theta functions and for $m=0$  it gives the Rogers' false theta function $\frac{1}{(q)^2_\infty} \sum_{n \geq 0} (-1)^n q^{n(n+1)/2}$.
We employ (\ref{false}) and include an additional summation variable to analyze $\dn_{2,2}(0,q)$ and $\dn_{2,3}(0,q)$:
\begin{align*}
& \sum_{m, n \geq 0} \frac{q^{mn+m+n}}{(q)^2_{m}(q)^2_{n}}
=\frac{1}{(q)_\infty}\sum_{n \geq 0} \frac{q^n}{(q)_{n}}  \sum_{m \geq 0} \frac{q^{m^2+(n+1)m}}{(q)_m^2} \\
&=\frac{1}{(q)_\infty} \sum_{m \geq 0} \frac{q^{m^2+m}}{(q)_m^2}\sum_{n \geq 0} \frac{q^{n+nm}}{(q)_n} =\frac{1}{(q)^2_\infty} \sum_{m \geq 0} \frac{q^{m^2+m}}{(q)_m}, 
\end{align*}
where in the last equality we use Euler's identity \eqref{eqn:euler1} with $z=q^{m+1}$. 
Along the same lines we get 
$$\sum_{m, n \geq 0} \frac{q^{mn+m+n+kn}}{(q)^2_{m}(q)^2_{n}}=\frac{1}{(q)^2_\infty} \sum_{m \geq 0} \frac{(q)_{m+k} q^{m^2+m}}{(q)^2_m},$$
which, after specialization $k=1$, gives 
$$\sum_{m,n \geq 0} \frac{q^{mn+m+2n}}{(q)^2_{m}(q)^2_{n}}=\frac{1}{(q)^2_\infty} \sum_{m \geq 0} \frac{(1-q^{m+1}) q^{m^2+m}}{(q)_m}.$$
This formula, the Rogers-Ramanujan recursion (\ref{RR-rec}) (specialized at $x=1$) together with the second Rogers-Ramanujan identity now gives (\ref{RR:non-vacuum}).

Next we consider the double pole series with three summation variables: $\dn_{3,3}(0,q)$, $\dn_{3,4}(0,q)$ and $\dn_{3,5}(0,q)$. We first compute
\begin{align*}
& \sum_{m,n,k \geq 0} \frac{q^{mn+nk+m+n+k}}{(q)_m^2 (q)_n^2 (q)_k^2}=\sum_{m \geq 0}  \frac{q^{m}}{(q)_m^2} \sum_{n,k \geq 0} \frac{q^{mn+nk+n+k}}{(q)_n^2 (q)_k^2} \\
&=\frac{1}{(q)^2_\infty} \sum_{m \geq 0}  \frac{q^{m}}{(q)_m^2}  \sum_{n \geq 0} \frac{(q)_{n+m}q^{n^2+n}}{(q)_n^2}=\frac{1}{(q)^2_\infty} \sum_{n \geq 0}  \frac{q^{n^2+n}}{(q)_n} \sum_{m \geq 0}  \frac{(q^{n+1})_{m} q^{m}}{(q)_m^2}=\frac{1}{(q)_\infty^3} \sum_{m,n \geq 0} \frac{q^{(m+n)^2+n^2+m+2n}}{(q)_n^2 (q)_m},
\end{align*}
where in the last equality we use (\ref{Jackson}) (with $a \to 0$, $b=q^{n+1}$, $c=z=q$), \eqref{eqn:pochA-B}, and we change summation variables. Completely analogously we get
$$\sum_{m,n,k \geq 0} \frac{q^{mn+nk+2m+n+k}}{(q)_m^2 (q)_n^2 (q)_k^2}=\frac{1-q}{(q)_\infty^3}\sum_{m,n \geq 0} \frac{q^{(m+n)^2+n^2+m+3n}}{(q)_n^2 (q)_m}.$$
Finally, 
\begin{align*}
& \sum_{m,n,k \geq 0} \frac{q^{mn+nk+m+2n+k}}{(q)_m^2 (q)_n^2 (q)_k^2}=\sum_{m \geq 0}  \frac{q^{m}}{(q)_m^2} \sum_{n,k \geq 0} \frac{q^{mn+nk+2n+k}}{(q)_n^2 (q)_k^2} =\frac{1}{(q)^2_\infty}  \sum_{m \geq 0}  \frac{q^{m}}{(q)_m^2}  \sum_{n \geq 0} \frac{(q)_{n+m+1}q^{n^2+n}}{(q)_n^2} \\
& =\frac{1}{(q)^2_\infty}\sum_{n \geq 0} \frac{(1-q^{n+1}) q^{n^2+n}}{(q)_n} \sum_{m \geq 0}  \frac{(q^{n+2})_{m}q^{m}}{(q)_m^2}=\frac{1}{(q)^3_\infty}\sum_{m, n \geq 0} \frac{(1-q^{n+m})^2 q^{(n+m)^2+n^2-m}}{(q)_n^2 (q)_m}
\end{align*}
where in the last equality we again use (\ref{Jackson}) (now with $a \to 0$, $b=q^{n+2}$, $c=z=q$), (\ref{eqn:pochA-B}), and we shift the summation variables.
These three identities combined with (\ref{false-rec}) (with $x=1$) and 
\begin{equation*}
\sum_{m, n \geq 0} \frac{ q^{(n+m)^2+n^2-m}}{(q)_n^2 (q)_m}=\sum_{m, n \geq 0} \frac{ q^{(n+m)^2+n^2}}{(q)_n^2 (q)_m}
+\sum_{m, n \geq 0} \frac{q^{(n+m)^2+n^2+2n+m}}{(q)_n^2 (q)_m}
\end{equation*}
give the following identities:
\begin{align*}
 (q)_\infty^3 \sum_{n_1,n_2,n_3 \geq 0} \frac{q^{n_1+n_2+n_3+n_1 n_2 +n_2 n_3}}{(q)_{n_1}^2 (q)_{n_2}^2 (q)_{n_3}^2} &= \sum_{n_1,n_2 \geq 0} \frac{q^{(n_1+n_2)^2+n_2^2+n_1+2n_2}}{(q)_{n_1}(q)_{n_2}^2} \\
 (q)_\infty^3 \sum_{n_1,n_2,n_3 \geq 0} \frac{(2-q^{n_1})q^{n_1+n_2+n_3+n_1 n_2 +n_2 n_3}}{(q)_{n_1}^2 (q)_{n_2}^2 (q)_{n_3}^2} &=\sum_{n_1,n_2 \geq 0} \frac{q^{(n_1+n_2)^2+n_2^2+n_2}}{(q)_{n_1}(q)_{n_2}^2}  \\
(q)_\infty^3 \sum_{n_1,n_2,n_3 \geq 0} \frac{(2-2q^{n_1}+q^{n_2})q^{n_1+n_2+n_3+n_1 n_2 +n_2 n_3}}{(q)_{n_1}^2 (q)_{n_2}^2 (q)_{n_3}^2} &=\sum_{n_1,n_2 \geq 0} \frac{q^{(n_1+n_2)^2+n_2^2}}{(q)_{n_1}(q)_{n_2}^2}.  \\
\end{align*}
Now the required statement (Theorem \ref{thm:doublepoleFalseAG}, $k=2$) follows 
from Theorem \ref{false-id}. We can proceed in this fashion to analyze $k$-fold summations, for $k \geq 4$, etc. We note that 
in this approach we naturally encounter more general $\dn$-type series 
	$$\sum\limits_{n_1,\dots,n_t\geq 0}\dfrac{(-w)_{n_1}q^{n_1n_2+\cdots+n_{t-1}n_t + a_1 n_1 + \cdots  +a_k n_k} }{(q)_{n_1}^2\cdots (q)_{n_t}^2}$$
where $a_i \in \mathbb{N}$. 
We plan to study new $q$-series identities for these series in our future publications.

\begin{rem}
Above, we have explicitly reduced our double-pole representations for Rogers-Ramanujan and false theta identities to their more well-known 
representations, namely, the right-hand side of \eqref{MFI:AG} with $k=1$ in the former case and \eqref{MFI:FT} with $k=2$ in the latter.
However, it is possible to prove double-pole representations for the Andrews-Gordon series in a different way
using a certain uniqueness property \cite{LepZhu}. 
One of us is currently investigating this \cite{Kan}.
\end{rem}

\bibliographystyle{abbrv}

\providecommand{\oldpreprint}[2]{\textsf{arXiv:\mbox{#2}/#1}}\providecommand{\preprint}[2]{\textsf{arXiv:#1
  [\mbox{#2}]}}

\end{document}